\documentclass[a4paper,leqno,10pt]{article}

\usepackage{epsf,graphicx,epsfig}
\usepackage{amscd}
\usepackage{amsmath,latexsym,amssymb,amsthm}
\usepackage[nospace,noadjust]{cite}
\usepackage{textcomp}
\usepackage{setspace,cite}
\usepackage{lscape,fancyhdr,fancybox}
\usepackage{stmaryrd}
\usepackage[all,cmtip]{xy}
\usepackage{tikz}
\usepackage{cancel}
\usetikzlibrary{shapes,arrows,decorations.markings}
\usepackage{graphicx}

\usepackage{color}
\usepackage{url}
\usepackage{enumerate}
\usepackage[mathscr]{euscript}

  \theoremstyle{plain}

\swapnumbers
    \newtheorem{thm}{Theorem}[section]
    \newtheorem{prop}[thm]{Proposition}

    \newtheorem{corollary}[thm]{Corollary}
    
    \newtheorem{subsec}[thm]{}
\theoremstyle{definition}
    \newtheorem{defn}[thm]{Definition}
        \newtheorem{remark}[thm]{Remark}
    \newtheorem{exam}[thm]{Example}

    \newtheorem{notation}[thm]{Notation}

\theoremstyle{remark}

\title{}
\author{}
\date{}
\usepackage{amssymb}

\usepackage{hyperref}
\hypersetup{
	colorlinks,
	citecolor=blue,
	filecolor=black,
	linkcolor=blue,
	urlcolor=black
}

\begin{document}

\title{Compatible $L_\infty$-algebras}

\author{Apurba Das \footnote{Department of Mathematics,
Indian Institute of Technology, Kharagpur 721302, West Bengal, India.} \footnote{Email: apurbadas348@gmail.com}}



\maketitle
\begin{abstract}
A compatible $L_\infty$-algebra is a graded vector space together with two compatible $L_\infty$-algebra structures on it. Given a graded vector space, we construct a graded Lie algebra whose Maurer-Cartan elements are precisely compatible $L_\infty$-algebra structures on it. We provide examples of compatible $L_\infty$-algebras arising from Nijenhuis operators, compatible $V$-datas and compatible Courant algebroids.
We define the cohomology of a compatible $L_\infty$-algebra and as an application, we study formal deformations. Next, we classify `strict' and `skeletal' compatible $L_\infty$-algebras in terms of crossed modules and cohomology of compatible Lie algebras. Finally, we introduce compatible Lie $2$-algebras and find their relationship with compatible $L_\infty$-algebras.
\end{abstract}

\medskip

\medskip





\medskip




\tableofcontents

\noindent
\thispagestyle{empty}


\section{Introduction}

Compatible algebraic and geometric structures appeared in various places of mathematics and mathematical physics. Among others, compatible Lie algebras are widely studied in the literature. A compatible Lie algebra is a triple $(\mathfrak{g}, [~,~], [~,~]')$ consisting of a vector space $\mathfrak{g}$ equipped with two Lie brackets $[~,~]$ and $[~,~]'$ satisfying some compatibility condition.
The compatibility is equivalent to the fact that the sum $[~,~] + [~,~]'$ is a Lie bracket on $\mathfrak{g}$. 
Compatible Lie algebras are closely related to Nijenhuis deformations of Lie algebras, classical Yang-Baxter equations, principal chiral fields, loop algebras over Lie algebras and elliptic theta functions \cite{golu1,golu2,golu3,nij-ric,Odesskii1}. Similarly, compatible associative algebras and their relations with associative Yang-Baxter equations, quiver representations and bialgebra theory are explored in \cite{mar,Odesskii2,Odesskii3,wu}. Recently, compatible Lie algebras and compatible associative algebras are studied in \cite{comp-lie,das-comp} from cohomological points of view. In the geometric context, compatible Poisson structures were appeared in the mathematical study of biHamiltonian mechanics \cite{kos,mag-mor}. Such geometric structures are closely related to compatible Lie algebras via dualization \cite{bol}. See \cite{dotsenko,stro} for the operadic study of compatible structures.

\medskip

The notion of $L_\infty$-algebras (also called strongly homotopy Lie algebras) are the homotopy analogue of Lie algebras \cite{jim,lada-markl}. More precisely, an $L_\infty$-algebra is a pair $(L, \{ l_k \}_{k \geq 1})$ consisting of a graded vector space $L$ together with a collection of skew-symmetric multilinear maps $l_k : L^{\otimes k} \rightarrow L$, for $k \geq 1$, of certain degrees that satisfy some higher Jacobi identities. $L_\infty$-algebras appeared in the study of deformation theory, multisymplectic geometry, Courant algebroids, Gauge theory and many classical topics in mathematics and mathematical physics \cite{l-def,l-mul,l-cou,l-gau,l-fie}. In \cite{baez-crans}, the authors introduced Lie $2$-algebras as the categorification of Lie algebras and showed that the category of Lie $2$-algebras is equivalent to the category of $L_\infty$-algebras with underlying graded vector space concentrated in two degrees.

\medskip

In this paper, we define a suitable compatibility condition between two $L_\infty$-algebra structures on a graded vector space. A compatible $L_\infty$-algebra is a triple $(L, \{l_k \}_{k \geq 1}, \{ l_k' \}_{k \geq 1})$ in which $(L, \{ l_k \}_{k \geq 1})$ and $(L, \{ l_k' \}_{k \geq 1})$ are $L_\infty$-algebras satisfying the compatibility condition. Compatible $L_\infty$-algebras can be seen as the homotopy analogue of compatible (graded) Lie algebras. Given a graded vector space $L$, we construct a graded Lie algebra whose Maurer-Cartan elements are in one-to-one correspondence with compatible $L_\infty$-algebra structures on $L$ (cf. Theorem \ref{comp-inf-gla}). We construct a compatible differential graded Lie algebra from a compatible Lie algebra via the Nijenhuis-Richardson bracket. We provide examples of compatible $L_\infty$-algebras arising from compatible $A_\infty$-algebras, Nijenhuis operators on $L_\infty$-algebras, compatible $V$-datas and compatible Courant algebroids. An application of compatible $V$-datas is given to relative Rota-Baxter operators on compatible Lie algebras.

\medskip

Next, we define the cohomology of a compatible $L_\infty$-algebra $(L, \{ l_k \}_{k \geq 1}, \{ l_k'\}_{k \geq 1})$ using the corresponding Maurer-Cartan element in the graded Lie algebra given in Theorem \ref{comp-inf-gla}. As an application of cohomology, we extensively study deformations of compatible $L_\infty$-algebras. We show that the vanishing of the second cohomology group of compatible $L_\infty$-algebra implies the rigidity of the structure (cf. Theorem \ref{2-rigid}). Moreover, given a finite order deformation, we construct a third cohomology class, called the obstruction class, which vanishes if and only if the given deformation extends to deformation of next order (cf. Theorem \ref{3-ext}).

\medskip

$L_\infty$-algebras whose underlying graded vector space is concentrated in two degrees (called $2$-term $L_\infty$-algebras) pay particular attention as they are related to the categorification of Lie algebras \cite{baez-crans}. In this paper, we focus on $2$-term compatible $L_\infty$-algebras. Some particular classes of such algebras are `strict algebras' and `skeletal algebras'. We introduce crossed module of compatible Lie algebras and show that there is a one-to-one correspondence between crossed modules of compatible Lie algebras and strict algebras (cf. Theorem \ref{strict-thm}). We also show that skeletal algebras can be classified by suitable $3$-cocycles of compatible Lie algebras (cf. Theorem \ref{skeletal-thm}). Finally, we introduce compatible Lie $2$-algebras as the categorification of compatible Lie algebras. We show that the category of compatible Lie $2$-algebras is equivalent to the category of compatible $L_\infty$-algebras whose underlying graded vector space is concentrated in two degrees and the differentials induced by two $L_\infty$-algebras are the same (cf. Theorem \ref{2-2thm}).

\medskip

The paper is organized as follows. In Section \ref{sec2}, we recall some basic results on compatible Lie algebras and $L_\infty$-algebras. In Section \ref{sec3}, we introduce compatible $L_\infty$-algebras as the homotopy analogue of compatible Lie algebras. Given a graded vector space $L$, we construct the graded Lie algebra whose Maurer-Cartan elements correspond to compatible $L_\infty$-algebra structures on $L$. Examples of compatible $L_\infty$-algebras from Nijenhuis operators, compatible $V$-datas and compatible Courant algebroids are given in Section \ref{sec4}. In Section \ref{sec5}, we introduce the cohomology of a compatible $L_\infty$-algebra and study deformations. Strict and skeletal compatible $L_\infty$-algebras and their classifications are considered in Section \ref{sec6}. Finally, in Section \ref{sec7}, we introduce compatible Lie $2$-algebras and prove Theorem \ref{2-2thm}.

\medskip

\noindent {\bf Notations.}

All (graded) vector spaces, (multi)linear maps, tensor products, wedge products are over a field {\bf k} of characteristic $0$. We denote the set of all permutations on $k$ elements by $S_k$. A permutation $\sigma \in S_k$ is called an $(i, k-i)$-shuffle if $\sigma (1) < \cdots < \sigma (i)$ and $\sigma (i+1) < \cdots < \sigma (k)$. Let $Sh (i, k-i)$ denotes the set of all $(i, k-i)$-shuffles. For any permutation $\sigma \in S_k$ and homogeneous elements $v_1, \ldots, v_k $ in a graded vector space $V$, the Koszul sign $\epsilon (\sigma) = \epsilon (\sigma ; v_1, \ldots, v_k)$ is given by $v_{\sigma (1)} \odot \cdots \odot v_{\sigma (k)} = \epsilon (\sigma) v_1 \odot \cdots \odot v_k$, where $\odot$ denotes the product on the symmetric algebra $SV$. A graded multilinear map $f: V^{\otimes k} \rightarrow V$ is called symmetric if $f (v_{\sigma (1)} , \ldots , v_{\sigma (k)}) = \epsilon (\sigma) f (v_1, \ldots, v_k)$ and called skew-symmetric if $f (v_{\sigma (1)} , \ldots , v_{\sigma (k)}) = (-1)^\sigma \epsilon (\sigma) f (v_1, \ldots, v_k)$, for all $\sigma \in S_k$ and homogeneous elements $v_1, \ldots, v_k \in V$.

\section{Preliminaries}\label{sec2}
In this section, we recall some basic results on compatible Lie algebras and $L_\infty$-algebras. Our main references are \cite{comp-lie,lada-markl,jim}.


\subsection*{Compatible Lie algebras}

\begin{defn}
A compatible Lie algebra is a triple $(\mathfrak{g}, [~,~], [~,~]')$ consisting of a vector space $\mathfrak{g}$ equipped with two Lie brackets $[~,~]$ and $[~,~]'$ on $\mathfrak{g}$ satisfying the compatibility
\begin{align}\label{comp-lie-def-id}
[[x,y]',z] + [[x,y],z]' = [[x,z]', y]+ [[x,z], y]' + [x, [y,z]'] + [x, [y,z]]', \text{ for } x, y, z \in \mathfrak{g}.
\end{align}
\end{defn}

Note that the compatibility condition (\ref{comp-lie-def-id}) is equivalent to say that $[~,~] + [~,~]'$ is also a Lie bracket on $\mathfrak{g}$.

\begin{defn}
Let $(\mathfrak{g}, [~,~], [~,~]')$ be a compatible Lie algebra. A representation of it consists of a triple $(V, \rho, \rho')$ in which $(V, \rho)$ is a representation of the Lie algebra $(\mathfrak{g}, [~,~])$, and $(V, \rho')$ is a representation of the Lie algebra $(\mathfrak{g}, [~,~]')$ satisfying additionally the following compatibility
\begin{align*}
\rho ([x, y]') + \rho' ([x,y]) = [\rho (x), \rho' (y)] + [\rho'(x), \rho (y)], ~ \text{ for } x, y \in \mathfrak{g}.
\end{align*}
\end{defn}

Any compatible Lie algebra $(\mathfrak{g}, [~,~], [~,~]')$ can be regarded as a representation $(\mathfrak{g}, \rho = \mathrm{ad}_{[~,~]}, \rho' = \mathrm{ad}_{[~,~]'})$ of itself. This is called the adjoint representation.

Let $(\mathfrak{g}, [~,~], [~,~]')$ be a compatible Lie algebra and $(V, \rho, \rho')$ be a representation of it. Consider the Chevalley-Eilenberg cochain complex $\{ \mathrm{Hom} (\wedge^\bullet \mathfrak{g}, V) , \partial \}$ of the Lie algebra $(\mathfrak{g}, [~,~])$ with coefficients in the representation $(V, \rho)$, where the  coboundary $\partial : \mathrm{Hom}(\wedge^n \mathfrak{g}, V) \rightarrow \mathrm{Hom}(\wedge^{n+1} \mathfrak{g}, V)$, for $n \geq 0$, is given by
\begin{align*}
(\partial f) (x_1, \ldots, x_{n+1}) =~& \sum_{i=1}^{n+1} (-1)^{i+1} \rho (x_i ) f (x_1, \ldots, \widehat{x_i}, \ldots, x_{n+1}) \\
~&+ \sum_{i < j } (-1)^{i+j} f ([x_i, x_j], x_1, \ldots, \widehat{x_i}, \ldots, \widehat{x_j}, \ldots, x_{n+1}),
\end{align*}
for $ f \in \mathrm{Hom} (\wedge^n \mathfrak{g}, V)$ and $x_1, \ldots, x_{n+1} \in \mathfrak{g}$. Similarly, we can consider the Chevalley-Eilenberg cochain complex $\{ \mathrm{Hom} (\wedge^\bullet \mathfrak{g}, V) , \partial' \}$ of the Lie algebra $(\mathfrak{g}, [~,~]')$ with coefficients in the representation $(V, \rho')$. It has been observed in \cite{comp-lie} that $\partial$ and $\partial'$ satisfies the compatibility $\partial \circ \partial' + \partial' \circ \partial = 0$. Therefore, we are now able to construct a new cochain complex $\{ C^\bullet_c (\mathfrak{g}, V), \partial_c \}$, where
\begin{align*}
C^0_c ( \mathfrak{g}, V ) = \{ v \in V |~ \rho (x) v = \rho' (x) v, \forall  x \in \mathfrak{g} \} ~~~ \text{ and } ~~~ C^{n \geq 1}_c (\mathfrak{g}, V) = \underbrace{ \mathrm{Hom}(\wedge^n \mathfrak{g}, V) \oplus \cdots \oplus \mathrm{Hom}(\wedge^n \mathfrak{g}, V)   }_{n \text{ times} }.
\end{align*}
The coboundary map $\partial_c : C^n_c (\mathfrak{g}, V) \rightarrow C^{n+1}_c (\mathfrak{g}, V)$, for $n \geq 0$, given by
\begin{align*}
\partial_c (f_1, \ldots, f_n ) = (\partial f_1, \ldots, \underbrace{\partial f_i + \partial' f_{i-1}}_{i\text{-th place}}, \ldots, \partial' f_n).
\end{align*}
The cohomology groups of the complex $\{ C^\bullet_c (\mathfrak{g}, V), \partial_c \}$ are called the cohomology of the compatible Lie algebra $(\mathfrak{g}, [~,~], [~,~]')$ with coefficients in the representation $(V, \rho, \rho')$.

\begin{remark}\label{rem-comp-co}
Let $n \geq 2$. Suppose $f, f' \in \mathrm{Hom}(\wedge^n \mathfrak{g}, V)$ are two elements that satisfies $\partial f = \partial' f' = \partial f' + \partial' f = 0$. Then the element $( f, \underbrace{f + f', \ldots, f+ f'}_{n-2 \text{ times}}, f') \in C^n_c (\mathfrak{g}, V)$ is a $n$-cocycle as
\begin{align*}
\partial_c ( f, f + f', \ldots, f+ f', f') = \big( \partial f, \partial f + \partial f' + \partial' f, \ldots, \underbrace{ \partial (f+f') + \partial' (f+f') }_{i\text{-th place}}, \ldots, \partial' f'  \big) = 0.
\end{align*}
Such $n$-cocycles are called induced by compatible Lie algebra cocycles. We will use such $3$-cocycles in Section \ref{sec6} to classify skeletal compatible $L_\infty$-algebras.
\end{remark}


\subsection*{$L_\infty$-algebras}

\begin{defn}
An $L_\infty$-algebra is a pair $(L, \{ l_k \}_{k \geq 1})$ consisting of a graded vector space $L= \oplus_{i \in \mathbb{Z}} L_i$ together with a collection $\{ l_k \}_{k \geq 1}$ of graded skew-symmetric multilinear maps $l_k : L^{\otimes k} \rightarrow L$ with $\text{deg } (l_k) = 2-k$, for $k \geq 1$, satisfying the following set of identities:
\begin{align}\label{hji}
\sum_{i+j = n+1} \sum_{\sigma \in Sh (i,n-i)} \chi (\sigma) (-1)^{i (j-1)} ~ l_j \big(  l_i (x_{\sigma (1)}, \ldots, x_{\sigma (i)}), x_{\sigma (i+1)}, \ldots, x_{\sigma (n)} \big) = 0,
\end{align}
for $n \geq 1$ and homogeneous elements $x_1, \ldots, x_n \in L.$
\end{defn}

The identities (\ref{hji}) are called the higher Jacobi identities. For small values on $n$, the meaning of the higher Jacobi identities can be found in \cite{lada-markl}.

\begin{remark}\label{rem-dgla}
Any differential graded Lie algebra $(L, \partial, [~,~])$ is an $L_\infty$-algebra in which $l_1 = \partial$, $l_2 = [~,~]$ and $l_k = 0$, for $k \geq 3$.
\end{remark}

Let $(L, \{ l_k \}_{k \geq 1})$ be an $L_\infty$-algebra. Take $V = s^{-1}L$ the desuspension of $L$, where $V_i = (s^{-1}L)_i = L_{i + 1}$, for $i \in \mathbb{Z}$. We define degree $1$ maps $\rho_k : V^{\otimes k} \rightarrow V$, for $k \geq 1$, by
\begin{align*}
\rho_k := (-1)^{\frac{k(k-1)}{2}}~ s^{-1} \circ l_k \circ s^{\otimes k }.
\end{align*}
It turns out that the map $\rho_k$ is graded symmetric, for $k \geq 1$. Note that the map $l_k$ can be reconstructed from $\rho_k$ by $l_k = s \circ \rho_k \circ (s^{-1})^{\otimes k}$.

For each $k \geq 1$, the map $\rho_k$ uniquely lifts to a degree $1$ coderivation $\widetilde{\rho_k} : SV \rightarrow SV$ on the symmetric coalgebra $SV$, given by
\begin{align*}
\widetilde{\rho_k} (v_1 \odot \cdots \odot v_n) = \sum_{\sigma \in Sh (k, n-k)} \epsilon (\sigma) ~ \rho_k ( v_{\sigma(1)} \odot \cdots \odot v_{\sigma(k)}) \odot v_{\sigma (k+1)} \odot \cdots \odot v_{\sigma (n)}, ~ \text{ for } n \geq k
\end{align*}
and zero otherwise. Note that the graded space of coderivations $\mathrm{Coder}^\bullet (SV)$ on the symmetric coalgebra $SV$ is a graded Lie algebra with the commutator bracket $[~,~]_C$.

\begin{prop}
An $L_\infty$-algebra structure on a graded vector space $L$ is equivalent to a Maurer-Cartan element in the graded Lie algebra $(\mathrm{Coder}^\bullet (S V) , [~,~]_C)$, where $V = s^{-1}L$.
\end{prop}

\section{Compatible $L_\infty$-algebras}\label{sec3}
In this section, we introduce compatible $L_\infty$-algebras as the homotopy analogue of compatible Lie algebras. Given a graded vector space $L$, we also construct a graded Lie algebra whose Maurer-Cartan elements are given by compatible $L_\infty$-algebra structures $L$.

\begin{defn}
A compatible $L_\infty$-algebra is a triple $(L, \{ l_k \}_{k \geq 1}, \{ l_k ' \}_{k \geq 1})$ in which  $(L, \{ l_k \}_{k \geq 1})$ and  $(L, \{ l_k' \}_{k \geq 1})$ are both $L_\infty$-algebras additionally satisfying
\begin{align}\label{hji-comp}
\sum_{i+j = n+1} \sum_{\sigma \in Sh (i, n-i)} \chi (\sigma) (-1)^{i (j-1)}~ \bigg\{& l_j \big(  l_i' (x_{\sigma (1)}, \ldots, x_{\sigma (i)}), x_{\sigma (i+1)}, \ldots, x_{\sigma (n)} \big) \\
&+ l_j' \big(  l_i (x_{\sigma (1)}, \ldots, x_{\sigma (i)}), x_{\sigma (i+1)}, \ldots, x_{\sigma (n)} \big) \bigg\} = 0, \nonumber
\end{align}
for $ n \geq 1$ and $x_1, \ldots, x_n \in L.$
\end{defn}

In this case, we often say that the $L_\infty$-algebras  $(L, \{ l_k \}_{k \geq 1})$ and $(L, \{ l_k' \}_{k \geq 1})$ are compatible.
Note that the compatibility condition (\ref{hji-comp}) is equivalent to the fact that the pair $(L, \{ l_k + l_k' \}_{k \geq 1})$ is an $L_\infty$-algebra.

\begin{exam}
(i) Any compatible Lie algebra is a compatible $L_\infty$-algebra whose underlying graded vector space is concentrated in degree $0$. Any compatible differential graded Lie algebra $(L, \{ \partial, [~,~] \}, \{ \partial', [~,~]' \})$ is also a compatible $L_\infty$-algebra with $l_k = l'_k = 0$, for $k \geq 3$. See also Proposition \ref{comp-dgla-prop}.\\

(ii) Let $(L, \{ l_k \}_{k \geq 1})$ be an $L_\infty$-algebra. Then $(L, \{l_k \}_{k \geq 1}, \{ l_k \}_{k \geq 1})$ is obviously a compatible $L_\infty$-algebra.\\

(iii) Let $(L, \{ l_k \}_{k \geq 1}, \{ l'_k \}_{k \geq 1})$ be a compatible $L_\infty$-algebra. Then $(L, \{ l'_k \}_{k \geq 1}, \{ l_k \}_{k \geq 1})$, $(L, \{ l_k \}_{k \geq 1}, \{ l_k + l'_k \}_{k \geq 1})$ and $(L, \{ l_k + l'_k \}_{k \geq 1}, \{ l'_k \}_{k \geq 1})$ are all compatible $L_\infty$-algebras.
\end{exam}

Some nontrivial examples of compatible $L_\infty$-algebras are provided in Sections \ref{sec4} and \ref{sec6}. In the next, we give some characterization results for compatible $L_\infty$-algebras.

\medskip


Let $(\mathfrak{g} = \oplus_{n \in \mathbb{Z}} \mathfrak{g}^n, [~,~]) $ be a graded Lie algebra and $\theta, \theta' \in \mathfrak{g}^1$ be two Maurer-Cartan elements. They are said to be compatible if they satisfies $[\theta, \theta'] = 0$. The following result gives a characterization of compatible $L_\infty$-algebra structures on a graded vector space $L$ in terms of compatible Maurer-Cartan elements in the graded Lie algebra $(\mathrm{Coder}^\bullet (SV), [~,~]_C)$, where $V = s^{-1}L$.

\begin{thm}\label{comp-inf-comp-mc}
There is a one-to-one correspondence between compatible $L_\infty$-algebra structures on a graded vector space $L$ and compatible Maurer-Cartan elements in the graded Lie algebra $(\mathrm{Coder}^\bullet (S V) , [~,~]_C)$, where $V = s^{-1}L$.
\end{thm}

\begin{proof}
Let $\{ l_k \}_{k \geq 1}$ and $\{ l'_k \}_{k \geq 1}$ be two collections of multilinear maps $l_k, l'_k : L^{\otimes k} \rightarrow L$ with $\mathrm{deg}(l_k) = \mathrm{deg}(l'_k) = 2-k$, for $k \geq 1$. We consider the collections $\{ \rho_k \}_{k \geq 1}$ and $\{ \rho'_k \}_{k \geq 1}$ of degree $1$ maps $\rho_k, \rho'_k : V^{\otimes k} \rightarrow V$ given by
\begin{align*}
\rho_k = (-1)^{\frac{k(k-1)}{2}} ~ s^{-1} \circ l_k \circ s^{\otimes k}   ~~~~ \text{ and } ~~~~ \rho'_k = (-1)^{\frac{k(k-1)}{2}} ~ s^{-1} \circ l'_k \circ s^{\otimes k}, ~\text{ for } k \geq 1.
\end{align*}
Define two degree $1$ coderivations $D, D' \in \mathrm{Coder}^1 (SV)$ by $D = \sum_{k \geq 1} \widetilde{\rho_k}$ and $D' = \sum_{k \geq 1} \widetilde{\rho'_k}$. Then we have
\begin{align*}
& [D,D]_C = 0 ~~~ \leftrightsquigarrow ~~~ (L, \{ l_k \}_{k \geq 1}) \text{ is an } L_\infty \text{-algebra},\\
& [D',D']_C = 0 ~~~ \leftrightsquigarrow ~~~ (L, \{ l'_k \}_{k \geq 1}) \text{ is an } L_\infty \text{-algebra},\\
& [D,D']_C = 0 ~~~ \leftrightsquigarrow ~~~ \text{the compatibility condition } (\ref{hji-comp}).
\end{align*}
This completes the proof.
\end{proof}

Our aim now is to construct a new graded Lie algebra whose Maurer-Cartan elements correspond to compatible $L_\infty$-algebra structures on $L$. We define a graded vector space $\mathrm{Coder}^\bullet_c (SV) = \oplus_{n \geq 0}\mathrm{Coder}^n_c (SV)$, where
\begin{align*}
\mathrm{Coder}^n_c (SV) := \underbrace{\mathrm{Coder}^n (SV) \oplus \cdots \oplus \mathrm{Coder}^n (SV)}_{n +1 \text{ times}}, \text{ for } n \geq 0.
\end{align*}
The graded Lie bracket $[~,~]_C$ on $\mathrm{Coder}^\bullet (SV)$ induces a bracket on $\mathrm{Coder}^\bullet_c (SV)$ given by
\begin{align*}
\llbracket (D_1, \ldots, D_{m+1}), &(D'_1, \ldots, D'_{n+1}) \rrbracket \\
& := \big(  [D_1, D_1']_C, \ldots, \underbrace{[D_1, D_i']_C + [D_2, D'_{i-1}]_C + \cdots + [D_i, D'_1]_C}_{i\text{-th place}}, \ldots, [D_{m+1}, D'_{n+1}]_C  \big),
\end{align*}
for $(D_1, \ldots, D_{m+1}) \in \mathrm{Coder}^m_c (SV)$ and $(D'_1, \ldots, D'_{n+1}) \in \mathrm{Coder}^n_c (SV)$. It is easy to verify that $\llbracket ~, ~ \rrbracket$ is a graded Lie bracket which makes $(\mathrm{Coder}^\bullet_c (SV), \llbracket ~, ~ \rrbracket )$ into a graded Lie algebra.

\begin{thm}\label{comp-inf-gla}
There is a one-to-one correspondence between compatible $L_\infty$-algebra structures on $L$ and Maurer-Cartan elements in the graded Lie algebra $(\mathrm{Coder}^\bullet_c (SV), \llbracket ~, ~ \rrbracket )$, where $V = s^{-1}L$.
\end{thm}

\begin{proof}
Let $D, D' \in \mathrm{Coder}^1 (SV)$ be two degree $1$ coderivations on $SV$. Consider the pair $(D,D') \in \mathrm{Coder}^1_c (SV)$. Then we have
\begin{align*}
\llbracket (D,D'), (D, D') \rrbracket = ( [D,D]_C, ~ [D,D']_C + [D', D]_C,~ [D', D']_C ).
\end{align*}
This shows that $D,D'$ are two compatible Maurer-Cartan elements in the graded Lie algebra $(\mathrm{Coder}^\bullet (SV), [~,~]_C)$ if and only if the pair $(D,D')$ is a Maurer-Cartan element in the graded Lie algebra $(\mathrm{Coder}^\bullet_c (SV), \llbracket ~, ~ \rrbracket)$. Hence by Theorem \ref{comp-inf-comp-mc}, we get
\begin{align*}
&\text{comp } L_\infty \text{ structure }  && \leftrightsquigarrow  &&& \text{comp Maurer-Cartan elements}  &&&& \leftrightsquigarrow   &&&&& \text{Maurer-Cartan element}\\
&(\{ l_k \}_{k \geq 1}, \{l'_k\}_{k \geq 1}) \text{ on } L  &&   &&& D, D' \in (\mathrm{Coder}^\bullet (SV), [~,~]_C) &&&&   &&&&& (D,D') \in (\mathrm{Coder}^\bullet_c (SV), \llbracket ~, ~ \rrbracket).
\end{align*}
This completes the proof. 
\end{proof}

\section{Examples of compatible $L_\infty$-algebras}\label{sec4}
In this section, we provide some examples of compatible $L_\infty$-algebras arising mainly from compatible $A_\infty$-algebras, Nijenhuis operators on $L_\infty$-algebras, compatible $V$-datas and compatible Courant algebroids. However, we start with an example of a compatible differential graded Lie algebra (dgLa) from a compatible Lie algebra.

\subsection*{Compatible dgLa induced from a compatible Lie algebra}

We have seen in Remark \ref{rem-dgla} that a differential graded Lie algebra can be considered as an $L_\infty$-algebra with trivial higher brackets. The following result is a consequence of the compatibility condition (\ref{hji-comp}).

\begin{prop}\label{comp-dgla-prop}
Let $(L, \partial. [~,~])$ and $(L, \partial', [~,~]')$ be two differential graded Lie algebras. Then the triple $(L, \{ \partial, [~,~]\}, \{ \partial', [~,~]' \})$ is a compatible differential graded Lie algebra if and only if
\begin{itemize}
\item $\partial \circ \partial' + \partial' \circ \partial = 0$,
\item $[[x, y], z]' + [[x,y]', z] = (-1)^{|y||z|} [[x, z], y]' + (-1)^{|y||z|} [[x,z]', y] + [x, [y,z]]' + [x, [y,z]'], $
\item $\partial [x, y]' + \partial' [x, y] = [\partial (x), y] ' + [\partial'(x), y] + (-1)^{|x|} [ x, \partial (y)]' + (-1)^{|x|} [x, \partial' (y)]$, ~ for $x, y, z \in \mathfrak{g}$. 
\end{itemize}
\end{prop}

In the following, we construct an example of a compatible differential graded Lie algebra from a compatible Lie algebra.

\begin{exam}
Let $(\mathfrak{g}, [~,~], [~,~]')$ be a compatible Lie algebra. Consider the graded vector space $L = \oplus_{n \geq 0} \mathrm{Hom} (\wedge^{n+1} \mathfrak{g}, \mathfrak{g})$ of all skew-symmetric multilinear maps on the vector space $\mathfrak{g}$. It carries a graded Lie bracket (Nijenhuis-Richardson bracket) given by
\begin{align*}
[f, g]_\mathsf{NR} :=~&  f \diamond g - (-1)^{mn} g \diamond f, ~ \text{ where }\\
(f \diamond g) (x_1, \ldots, x_{m+n+1}) =~& \sum_{\sigma \in Sh (n+1, m)} (-1)^\sigma ~ f \big(  g( x_{\sigma (1)},\ldots, x_{\sigma (n+1)} ), x_{\sigma (n+2)}, \ldots, x_{\sigma (m+n+1)} \big),
\end{align*}
for $f \in \mathrm{Hom} (\wedge^{m+1} \mathfrak{g}, \mathfrak{g})$ and $g \in \mathrm{Hom} (\wedge^{n+1} \mathfrak{g}, \mathfrak{g})$. Let the compatible Lie brackets $[~,~]$ and $[~,~]'$ correspond to elements $\mu$ and $\mu'$ in $L_1 = \mathrm{Hom} (\wedge^2 \mathfrak{g}, \mathfrak{g})$. Then $\mu$ and $\mu'$ satisfies 
\begin{align*}
[\mu , \mu]_\mathsf{NR} = 0, ~\quad   ~~~ [\mu', \mu' ]_\mathsf{NR} = 0 ~~~~ \text{ and } ~~~~  [\mu, \mu']_\mathsf{NR} = 0.
\end{align*}
This implies that the differential graded Lie algebras $(L, \partial_\mu = [\mu, -]_\mathsf{NR}, [~,~]_\mathsf{NR}, )$ and $(L, \partial_{\mu'} = [\mu', -]_\mathsf{NR}, [~,~]_\mathsf{NR}, )$ satisfy the conditions of Proposition \ref{comp-dgla-prop}. In other words, $(L, \{ \partial_\mu, [~,~]_\mathsf{NR} \}, \{ \partial_{\mu'} , [~,~]_\mathsf{NR} \})$ is a compatible differential graded Lie algebra.
\end{exam}

\subsection*{Compatible $A_\infty$-algebras}

The notion of $A_\infty$-spaces and $A_\infty$-algebras (strongly homotopy associative algebras) was introduced by Stasheff in the study of group-like topological spaces \cite{stas}. They have important applications in homotopy theory and combinatorial study of associahedra.

\begin{defn}
An $A_\infty$-algebra is a pair $(A, \{ \mu_k \}_{k \geq 1})$ consisting of a graded vector space $A = \oplus_{i \in \mathbb{Z}} A_i$ together with a collection $\{ \mu_k \}_{k \geq 1} $ of graded multilinear maps $\mu_k : A^{\otimes k} \rightarrow A$ with $\mathrm{deg}(\mu_k) = 2-k$, for $k \geq 1$, satisfying for $n \geq 1$ and homogeneous elements $a_1, \ldots, a_n \in A$,
\begin{align*}
\sum_{i+j = n+1} \sum_{\lambda =1}^j (-1)^{\lambda (i+1) + i(|a_1|+ \cdots + |a_{i-1}|)}  \mu_j \big( a_1 , \ldots, a_{i-1}, \mu_i (a_i, \ldots, a_{i+k-1}), \ldots, a_n \big) = 0.
\end{align*}
\end{defn}

Let $W = s^{-1} A$. Note that a collection $\{ \mu_k \}_{k \geq 1}$ of graded multilinear maps $\mu_k : A^{\otimes k} \rightarrow A$ with $\mathrm{deg}(\mu_k) = 2-k$ induces a collection $\{ \tau_k : W^{\otimes k} \rightarrow W \}_{k \geq 1}$ of degree $1$ multilinear maps, where $\tau_k = (-1)^{\frac{k(k-1)}{2}} s^{-1} \circ \mu_k \circ s^{\otimes k}$, for $k \geq 1$. Hence we obtain a degree $1$ coderivation $E = \sum_{k=1}^\infty \widetilde{\tau_k} : T W \rightarrow TW$ on the tensor coalgebra $TW$, where $\widetilde{\tau_k} : TW \rightarrow TW$ is given by
\begin{align*}
\widetilde{\tau_k} (w_1 \otimes \cdots \otimes w_n) = \sum_{i=1}^{n-k+1} (-1)^{|w_1| + \cdots + |w_{i-1}|}~ w_1 \otimes \cdots \otimes \tau_k (w_i \otimes \cdots \otimes w_{i+k-1}) \otimes \cdots \otimes w_n,
\end{align*}
for $n \geq k$ and zero otherwise. Consider the graded Lie algebra $(\mathrm{Coder}^\bullet (TW), [~,~]_C)$ on the graded vector space $\mathrm{Coder}^\bullet (TW)$ of coderivations on the tensor coalgebra $TW$. Then $(A, \{ \mu_k \}_{k \geq 1})$ is an $A_\infty$-algebra if and only if $E \in \mathrm{Coder}^1 (TW)$ satisfies $[E,E]_C = 0$. The following result has been proved in \cite{lada-markl}.

\begin{prop}\label{skew-inf} Let $(A, \{ \mu_k \}_{k \geq 1})$ be an $A_\infty$-algebra. Then $(A, \{ l_k \}_{k \geq 1})$ is an $L_\infty$-algebra, where
\begin{align*}
l_k (a_1, \ldots, a_k ) = \sum_{\sigma \in S_k} (-1)^\sigma \mu_k (a_{\sigma (1)}, \ldots, a_{\sigma (k)}), ~ \text{ for } k \geq 1.
\end{align*}
The collection $\{ l_k \}_{k \geq 1}$ is called the skew-symmetrization of $\{ \mu_k \}_{k \geq 1}$.
\end{prop}

\begin{defn}
A compatible $A_\infty$-algebra is a triple $(A, \{ \mu_k \}_{k \geq 1}, \{ \mu'_k \}_{k \geq 1})$ in which $(A, \{ \mu_k \}_{k \geq 1})$, $(A, \{ \mu'_k \}_{k \geq 1})$ and $(A, \{ \mu_k + \mu'_k \}_{k \geq 1})$ are all $A_\infty$-algebras.
\end{defn}

It is easy to prove that (similar to Theorem \ref{comp-inf-gla}) a compatible $A_\infty$-algebra structure on a graded vector space $A$ is equivalent to a Maurer-Cartan element in the graded Lie algebra $(\mathrm{Coder}^\bullet_c (TW), \llbracket ~, ~ \rrbracket)$, where $W = s^{-1}A$.

\begin{prop}
Let $(A, \{ \mu_k \}_{k \geq 1}, \{ \mu'_k \}_{k \geq 1})$ be a compatible $A_\infty$-algebra. Then $(A, \{ l_k \}_{k \geq 1}, \{ l'_k \}_{k \geq 1})$ is a compatible $L_\infty$-algebra, where $\{ l_k \}_{k \geq 1}$ and $\{ l'_k \}_{k \geq 1}$ are the skew-symmetrizations of $\{ \mu_k \}_{k \geq 1}$ and $\{ \mu'_k\}_{k \geq 1}$, respectively.
\end{prop}

\begin{proof}
Since $(A, \{ \mu_k \}_{k \geq 1})$ and $(A, \{ \mu'_k \}_{k \geq 1})$ are $A_\infty$-algebras, it follows that $(A, \{ l_k \}_{k \geq 1})$ and $(A, \{ l'_k \}_{k \geq 1})$ are $L_\infty$-algebras. Further, $(A, \{ \mu_k + \mu'_k \}_{k \geq 1})$ is an $A_\infty$-algebra implies that $(A, \{ l_k + l'_k \}_{k \geq 1})$ is an $L_\infty$-algebra (by Proposition \ref{skew-inf}). Thus, we get that $(A, \{ l_k \}_{k \geq 1}, \{ l'_k \}_{k \geq 1})$ is a compatible $L_\infty$-algebra.
\end{proof}

\subsection*{Nijenhuis operators on $L_\infty$-algebras}
In this subsection, we state the main result from \cite{azimi} about Nijenhuis operators on $L_\infty$-algebras, and observe that Nijenhuis operators induce compatible $L_\infty$-algebras. Let $(L, \{ l_k \}_{k \geq 1})$ be an $L_\infty$-algebra. Consider the corresponding Maurer-Cartan element $D \in \mathrm{Coder}^1(SV)$ in the graded Lie algebra $(\mathrm{Coder}^\bullet (SV), [~,~]_C)$, where $V =s^{-1} L$.

Let $\{ N_k : L^{\otimes k} \rightarrow L \}_{k \geq 1}$ be a collection of skew-symmetric multilinear maps with $\mathrm{deg}(l_k) = 1-k$, for $k \geq 1$. Consider the collection $\{ \mathcal{N}_k : V^{\otimes k} \rightarrow V \}_{k \geq 1}$ of degree $0$ symmetric multilinear maps on $V$ given by $\mathcal{N}_k = (-1)^{\frac{k(k-1)}{2}} s^{-1} \circ N_k \circ s^{\otimes k}$, for $k \geq 1$. This collection of symmetric multilinear maps lifts to a degree $0$ coderivation $N = \sum_{k \geq 1} \widetilde{\mathcal{N}_k} \in \mathrm{Coder}^0(SV)$ on the symmetric coalgebra $SV$.

\begin{defn}
A collection $\{ N_k : L^{\otimes k} \rightarrow L\}_{k \geq 1}$ of skew-symmetric multilinear maps (with $\mathrm{deg}(N_k) = 1-k$, for $k \geq 1$) is said to be a Nijenhuis operator on the $L_\infty$-algebra $(L, \{l_k\}_{k \geq 1})$ if there exists a degree $0$ coderivation ${P} \in \mathrm{Coder}^0(SV)$ such that
\begin{align*}
[N, [N, D]_C ]_C = [P, D]_C ~~~~ \text{ and } ~~~~ [N,P]_C =0.
\end{align*}
\end{defn}

Such a $P$ is called a square of $N$. The following result has been proved in \cite[Theorem 2.4]{azimi}.

\begin{prop}
Let $(L, \{l_k\}_{k \geq 1})$ be an $L_\infty$-algebra with the corresponding Maurer-Cartan element $D \in \mathrm{Coder}^1 (SV)$ in the graded Lie algebra $(\mathrm{Coder}^\bullet (SV), [~,~]_C)$, where $V = s^{-1} L$. Suppose $\{ N_k : L^{\otimes k} \rightarrow L\}_{k \geq 1}$ is a Nijenhuis operator with the corresponding coderivation $N \in \mathrm{Coder}^0 (SV)$. Then the coderivation $D_N := [N, D]_C \in \mathrm{Coder}^1 (SV)$ satisfies $[D_N, D_N]_C = 0$ and $[D,D_N ]_C = 0$. Hence $D_N$ corresponds to an $L_\infty$-algebra structure $(L, \{ l_k^N \}_{k \geq 1})$ on the same graded vector space $L$ and $(L, \{l_k \}_{k \geq 1}, \{l_k^N \}_{k \geq 1})$ is a compatible $L_\infty$-algebra.
\end{prop}

\subsection*{Compatible $V$-datas}

Let $L$ be a graded Lie algebra (with the bracket $[~,~]$), $\mathfrak{a} \subset L$ be an abelian Lie subalgebra, $P : L \rightarrow L$ be a projection such that $\mathrm{im }(P) = \mathfrak{a}$ and $\mathrm{ker }(P) \subset L$ be a Lie subalgebra. Let $\triangle \in \mathrm{ker }(P)_1$ be an element that satisfies $[\triangle, \triangle] = 0$. Such an element is called a `square-zero odd element'. In this case, the quadruple $(L, \mathfrak{a}, P, \triangle)$ is called a $V$-data \cite{voro}.

Given a $V$-data $(L, \mathfrak{a}, P, \triangle)$, the graded vector space $\mathfrak{a}$ carries some degree $1$ symmetric multilinear maps given by
\begin{align*}
\rho_k (a_1, \ldots, a_k ) := P [ \cdots [ [ \triangle, a_1 ], a_2], \ldots, a_k ], \text{ for } k \geq 1.
\end{align*}
Moreover, the pair $(s \mathfrak{a}, \{ l_k \}_{k \geq 1})$ is an $L_\infty$-algebra \cite{voro}, where $s \mathfrak{a}$ is the suspension of $\mathfrak{a}$ and $l_k$'s are given by $l_k = s \circ \rho_k \circ (s^{-1})^{\otimes k},$ for $k \geq 1$.

\begin{defn}
Let $(L, \mathfrak{a}, P, \triangle)$ and $(L, \mathfrak{a}, P, \triangle')$ be two $V$-datas with the same underlying structures except  the square-zero odd elements. They are said to be compatible if $[\triangle, \triangle'] = 0$.
\end{defn}

\begin{prop}\label{comp-v-comp-l}
Let $(L, \mathfrak{a}, P, \triangle)$ and $(L, \mathfrak{a}, P, \triangle')$ be two compatible $V$-datas with the corresponding $L_\infty$-algebras $(s\mathfrak{a}, \{ l_k \}_{k \geq 1})$ and $(s\mathfrak{a}, \{ l'_k \}_{k \geq 1})$. Then $(s \mathfrak{a}, \{ l_k \}_{k \geq 1}, \{ l'_k \}_{k \geq 1})$ is a compatible $L_\infty$-algebra.
\end{prop}

\begin{proof}
Since $(L, \mathfrak{a}, P, \triangle)$ and $(L, \mathfrak{a}, P, \triangle')$ are two compatible $V$-datas, it follows that $(L, \mathfrak{a}, P, \triangle + \triangle')$ is a $V$-data. If $(s \mathfrak{a}, \{ l^*_k \}_{k \geq 1})$ is the $L_\infty$-algebra corresponding to the $V$-data $(L, \mathfrak{a}, P, \triangle + \triangle')$, then it is easy to see that $l^*_k = l_k + l'_k$, for $k \geq 1$. This implies that $(s \mathfrak{a}, \{ l_k\}_{k \geq 1}, \{l'_k \}_{k \geq 1})$ is a compatible $L_\infty$-algebra.
\end{proof}

In the following, we give an example of compatible $V$-datas and the use of the corresponding compatible $L_\infty$-algebra. Let $\mathfrak{g}$ and $V$ be two vector spaces. Consider the graded Lie algebra 
\begin{align*}
L = \big( \oplus_{n \geq 0} \mathrm{Hom} (\wedge^{n+1} (\mathfrak{g} \oplus V), \mathfrak{g} \oplus V), [~,~]_\mathsf{NR}  \big)
\end{align*}
on the space of all skew-symmetric multilinear maps on $\mathfrak{g} \oplus V$ with the Nijenhuis-Richardson bracket. Then $\mathfrak{a} = \oplus_{n \geq 0} \mathrm{Hom} (\wedge^{n+1} V, \mathfrak{g}) \subset L$ is an abelian Lie subalgebra. Let $P : L \rightarrow L$ be the projection onto the subspace $\mathfrak{a}$. Suppose there are two skew-symmetric bilinear brackets $[~,~], [~,~]' : \mathfrak{g} \times \mathfrak{g} \rightarrow \mathfrak{g}$ and bilinear maps $\rho , \rho' : \mathfrak{g} \times V \rightarrow V$. Consider the elements
$\triangle , \triangle' \in \mathrm{Hom}(\wedge^2 (\mathfrak{g} \oplus V), \mathfrak{g} \oplus V)$ defined respectively by
\begin{align*}
\triangle ((x,u), (y,v)) = ([x,y], \rho (x, v) - \rho (y) u) ~ \text{ and } ~
\triangle' ((x,u), (y,v)) = ([x,y]', \rho' (x) v - \rho' (y) u),
\end{align*}
for $(x,u), (y, v) \in \mathfrak{g} \oplus V$. Then we have the following.

\begin{prop}
With the above notations, $(\mathfrak{g}, [~,~], [~,~]')$ is a compatible Lie algebra and $(V, \rho, \rho')$ is a representation if and only if $[\triangle, \triangle]_\mathsf{NR} = 0$, $[\triangle', \triangle']_\mathsf{NR} = 0$ and $[\triangle, \triangle']_\mathsf{NR} = 0$.
\end{prop}

\begin{proof}
For any $(x, u), (y, v), (z, w) \in \mathfrak{g} \oplus V$, we have
\begin{align*}
&[\triangle, \triangle]_\mathsf{NR} ( (x, u), (y, v), (z, w) )  \\
&= 2 \big( [[x, y],z] - [[x,z], y]+[[y,z],x],~ \rho ([x,y])w - [\rho(x), \rho(y)]w - \rho([x, z])v + [\rho (x), \rho (z)] v \\
& \qquad \qquad \qquad \qquad \qquad \qquad \qquad \qquad + \rho ([y,z])u - [\rho (y), \rho(z)] u   \big),\\\\
&[\triangle', \triangle']_\mathsf{NR} ( (x, u), (y, v), (z, w) )  \\
&= 2 \big( [[x, y]',z]' - [[x,z]', y]'+[[y,z]',x]',~ \rho' ([x,y]')w - [\rho'(x), \rho'(y)]w - \rho'([x, z]')v + [\rho' (x), \rho' (z)] v \\
& \qquad \qquad \qquad \qquad \qquad \qquad \qquad \qquad + \rho' ([y,z]')u - [\rho' (y), \rho'(z)] u   \big)
\end{align*}
and
\begin{align*}
&[\triangle, \triangle']_\mathsf{NR} ( (x, u), (y, v), (z, w) )  \\
&= \big( [[x, y]',z] - [[x,z]', y]+[[y,z]',x] + [[x, y],z]' - [[x,z], y]'+[[y,z],x]', \\
& \quad \quad \rho ([x,y]')w - [\rho(x), \rho'(y)]w - \rho([x, z]')v + [\rho (x), \rho' (z)] v + \rho ([y,z]')u - [\rho (y), \rho'(z)] u \\
&\quad \quad + \rho' ([x,y])w - [\rho'(x), \rho(y)]w - \rho'([x, z])v + [\rho' (x), \rho (z)] v + \rho' ([y,z])u - [\rho' (y), \rho(z)] u   \big).
\end{align*}
This shows that $(\mathfrak{g}, [~,~], [~,~]')$ is a compatible Lie algebra and $(V, \rho, \rho')$ is a representation if and only if $[\triangle, \triangle]_\mathsf{NR} = 0$, $[\triangle', \triangle']_\mathsf{NR} = 0$ and $[\triangle, \triangle']_\mathsf{NR} = 0$.
\end{proof}

\medskip

Let $(\mathfrak{g}, [~,~], [~,~]')$ be a compatible Lie algebra and $(V, \rho, \rho')$ be a representation. Then it follows from the above Proposition that $(L, \mathfrak{a}, P, \triangle)$ and $(L, \mathfrak{a}, P, \triangle')$ are compatible $V$-datas. Therefore, the space $s \mathfrak{a} = \oplus_{n \geq 0} \mathrm{Hom}(\wedge^n V, \mathfrak{g})$ carry compatible graded Lie brackets (as higher brackets are zero for degree reason) given by
\begin{align*}
\lbrace \!  [ P, Q ] \! \rbrace = (-1)^m [[\triangle, P ]_\mathsf{NR}, Q]_\mathsf{NR} ~~~ \text{ and } ~~~ \lbrace \!  [ P, Q ] \! \rbrace' = (-1)^m [[\triangle', P ]_\mathsf{NR}, Q]_\mathsf{NR},
\end{align*}
for $P \in (s \mathfrak{a})_m = \mathrm{Hom}(\wedge^m V, \mathfrak{g})$ and $Q \in (s \mathfrak{a})_n = \mathrm{Hom}(\wedge^n V, \mathfrak{g})$. In other words, $(s \mathfrak{a}, \lbrace \!  [ ~, ~ ] \! \rbrace , \lbrace \!  [ ~, ~ ] \! \rbrace')$ is a compatible graded Lie algebra. We will use this graded Lie algebra to characterize relative Rota-Baxter operators on compatible Lie algebras.

\begin{defn}
Let $(\mathfrak{g}, [~,~], [~,~]')$ be a compatible Lie algebra and $(V, \rho, \rho')$ be a representation. A linear map $R : V \rightarrow \mathfrak{g}$ is said to be a relative Rota-Baxter operator if it satisfies
\begin{align*}
[R (v) , R(w)] = R \big( \rho (R (v)) w - \rho (R(w)) v \big) ~~ \text{ and } ~~ [R (v) , R(w)]' = R \big( \rho' (R (v)) w - \rho' (R(w)) v \big), ~ \text{ for } v, w \in V.
\end{align*}
\end{defn}

We have the following result about the characterization of relative Rota-Baxter operators.

\begin{prop}
A linear map $R: V \rightarrow \mathfrak{g}$ is a relative Rota-Baxter operator if and only if $R$ is a Maurer-Cartan element in the compatible graded Lie algebra $(s\mathfrak{a}, \lbrace \!  [ ~, ~ ] \! \rbrace , \lbrace \!  [ ~, ~ ] \! \rbrace')$, i.e. $R$ satisfies $\lbrace \!  [ R,R ] \! \rbrace = \lbrace \!  [ R,R ] \! \rbrace' = 0$.
\end{prop}

\begin{proof}
By a straightforward calculation (see \cite{tang-o}), we get that
\begin{align*}
\lbrace \!  [ R,R ] \! \rbrace (u,v) =~& 2 \big(  R (\rho (R(v)) w) - R (\rho (R(w)) v)  - [R(v), R(w)]  \big),\\
\lbrace \!  [ R,R ] \! \rbrace' (u,v) =~& 2 \big(  R (\rho' (R(v)) w) - R (\rho' (R(w)) v)  - [R(v), R(w)]' \big),
\end{align*}
for $v,w \in V$. This shows that $R$ is a relative Rota-Baxter operator if and only if $\lbrace \!  [ R,R ] \! \rbrace  = \lbrace \!  [ R,R ] \! \rbrace' = 0$.
\end{proof}

\subsection*{Compatible Courant algebroids}

Courant algebroids were introduced in \cite{liu-cou} as a double object of Lie bialgebroids. A Courant algebroid over a smooth manifold $M$ consists of a quadruple $(E,  \langle ~,~ \rangle, [~,~]_E, \rho)$ of a smooth vector bundle $E$ over $M$, a symmetric nondegenerate $C^\infty (M)$-valued pairing $\langle ~, ~ \rangle : \Gamma E \times \Gamma E \rightarrow C^\infty (M)$, a bilinear bracket $[~,~]_E : \Gamma E \times \Gamma E$ and a bundle map $\rho : E \rightarrow TM$ satisfying the following set of identities:
\begin{itemize}
\item $[X, [Y,Z]_E ]_E = [[X,Y]_E,Z]_E + [Y, [X,Z]_E ]_E$ ~~~ (left Leibniz identity),
\item $[X, f Y]_E = f [X, Y]_E + \rho (X) (f ) Y,$
\item $\rho [X,Y]_E = [\rho (X), \rho (Y)],$
\item $[X,X]_E = \frac{1}{2} \mathcal{D} \langle X, X \rangle,$
\item $\rho (X) \langle Y, Z \rangle = \langle [X,Y]_E, Z \rangle + \langle Y, [X,Z]_E \rangle,$
\end{itemize}
for $X,Y,Z \in \Gamma E$ and $f \in C^\infty (M)$. Here $\mathcal{D} : C^\infty (M) \rightarrow \Gamma E$ is the composition $C^\infty (M) \xrightarrow{ d } \Omega^1 (M) \xrightarrow{\rho^*} \Gamma E^* \xrightarrow{ \simeq } \Gamma E$, where the last isomorphism is induced by the pairing $\langle ~, ~ \rangle$.

In \cite{l-cou} Roytenberg and Weinstein showed that a Courant algebroid $(E,  \langle ~,~ \rangle, [~,~]_E, \rho)$  associates an $L_\infty$-algebra on the graded vector space $L (E) := C^\infty (M) \oplus \Gamma E$ ~ (where $C^\infty (M)$ is concentrated in degree $-1$ and $\Gamma E$ is concentrated in degree $0$) whose structure maps are given by
\begin{align*}
&l_1 = \mathcal{D},\\
&l_2 (X,Y) = \lfloor X , Y \rfloor_E := [X,Y]_E - \frac{1}{2} \mathcal{D} \langle X, Y \rangle,\\
&l_2 (X, f ) =  - l_2 (f, X) = \frac{1}{2} \langle X, \mathcal{D} f \rangle,\\
&l_3 (X, Y, Z) = - \frac{1}{6} \langle \lfloor X, Y \rfloor_E , Z \rangle + c.p.
\end{align*}
(and the higher multilinear maps are vanish), for $X,Y,Z \in \Gamma E$ and $f \in C^\infty (M)$. Here we followed the convension of \cite{roy-thesis}. We denote this $L_\infty$-algebra by $(L(E), l_1, l_2, l_3)$.

\begin{defn}
Two Courant algebroids $(E,  \langle ~,~ \rangle, [~,~]_E, \rho)$  and $(E,  \langle ~,~ \rangle, [~,~]'_E, \rho')$ on the underlying same vector bundle $E$ and pairing $\langle ~, ~ \rangle$ are said to be compatible if the following conditions are hold:
\begin{align}
[X, [Y,Z]_E ]'_E + [X, [Y,Z]'_E ]_E =~& [[X,Y]_E,Z]'_E + [[X,Y]'_E,Z]_E + [Y, [X,Z]_E ]'_E + [Y, [X,Z]'_E ]_E, \label{cou-comp1}\\
\rho ([X, Y]_E') + \rho' ([X,Y]_E) =~& [ \rho (X), \rho' (Y)] + [\rho' (X), \rho (Y)], \label{cou-comp2}
\end{align} 
for $X,Y,Z \in \Gamma E$. Note that the conditions (\ref{cou-comp1}) and (\ref{cou-comp2}) is equivalent to the fact that the new quadruple $(E, \langle ~, ~\rangle, [~,~]_E + [~,~]_E', \rho + \rho')$ is also a Courant algebroid.
\end{defn}

In \cite{yks}, Kosmann-Schwarzbach introduced Nijenhuis operators on a given Courant algebroid and show that certain Nijenhuis operators induced a deformed Courant algebroid structure. Moreover, the deformed one is compatible with the given one. In the following, we prove that compatible Courant algebroids give rise to compatible $L_\infty$-algebras.

\begin{prop}
Let $(E, \langle ~,~ \rangle,  [~,~]_E, \rho)$  and $(E, \langle ~,~ \rangle,  [~,~]'_E, \rho')$ be two compatible Courant algebroids. Consider the corresponding $L_\infty$-algebras $(L(E), l_1, l_2, l_3)$ and $(L(E), l'_1, l'_2, l'_3)$ on the graded vector space $L(E)$. Then $(L(E), \{ l_1, l_2, l_3 \}, \{ l'_1, l'_2, l'_3 \})$ is a compatible $L_\infty$-algebra.
\end{prop}

\begin{proof}
Let $(L(E), l_1^*, l_2^*, l_3^*)$ be the $L_\infty$-algebra associated to the Courant algebroid $(E, [~,~]_E + [~,~]_E', \langle ~, ~\rangle, \rho + \rho')$. It is then easy to see that 
\begin{align}\label{comp-cou-id}
l_1^* = l_1 + l'_1, ~~~~ l_2^* = l_2 + l'_2 ~~ \text{ and } ~~ l_3^* = l_3 + l'_3.
\end{align}
It follows from (\ref{comp-cou-id}) that $(L(E), \{ l_1, l_2, l_3 \}, \{ l'_1, l'_2, l'_3 \})$ is a compatible $L_\infty$-algebra.
\end{proof}

\section{Cohomology and deformations of compatible $L_\infty$-algebras}\label{sec5}
In this section, we first introduce the cohomology of a compatible $L_\infty$-algebra. Then we study formal and finite order deformations of a compatible $L_\infty$-algebra in terms of cohomology.

\subsection*{Cohomology of compatible $L_\infty$-algebras}
 Let $(L, \{l_k \}_{k \geq 1}, \{ l_k' \}_{k \geq 1})$ be a compatible $L_\infty$-algebra. In Theorem \ref{comp-inf-gla}, we have seen that this compatible $L_\infty$-algebra structure on $L$ corresponds to a Maurer-Cartan element $(D,D') \in \mathrm{Coder}^1_c (SV)$ in the graded Lie algebra $(\mathrm{Coder}^\bullet_c (SV), \llbracket ~, ~ \rrbracket )$, where $V = s^{-1} L$. We will use this Maurer-Cartan element to define the cohomology of the compatible $L_\infty$-algebra $(L, \{l_k \}_{k \geq 1}, \{ l_k' \}_{k \geq 1})$.
 

For each $n \geq 1$, we define the $n$-th cochain group $C^n_c (L,L)$ by
\begin{align*}
C^n_c (L, L) := \mathrm{Coder}^{(n-1)}_c (SV) = \underbrace{\mathrm{Coder}^{(n-1)} (SV) \oplus \cdots \oplus \mathrm{Coder}^{(n-1)} (SV)}_{n \text{ times}},
\end{align*}
and a map $\delta_c : C^n_c (L, L) \rightarrow C^{n+1}_c (L, L)$, for $n \geq 1$, defined by $\delta_c ( - ) := \llbracket (D,D'), - \rrbracket.$
Then it is easy to see that $(\delta_c)^2 = 0$. In other words, $\{ C^\bullet_c (L,L), \delta_c \}$ is a cochain complex.
The corresponding cohomology groups are called the cohomology of the compatible $L_\infty$-algebra $L$, and they are denoted by $H^\bullet_c (L,L).$

\subsection*{Deformations of compatible $L_\infty$-algebras}

Let $L = (L, \{l_k \}_{k \geq 1}, \{ l_k' \}_{k \geq 1})$ be a compatible $L_\infty$-algebra with the corresponding Maurer-Cartan element $(D,D') \in \mathrm{Coder}^{1}_c (SV) = \mathrm{Coder}^{1} (SV) \oplus \mathrm{Coder}^{1} (SV)$ in the graded Lie algebra $(\mathrm{Coder}^\bullet_c (SV), \llbracket ~, ~ \rrbracket)$.

\begin{defn}
A formal deformation of the compatible $L_\infty$-algebra $L$ consists of a formal power series
\begin{align}\label{for-def}
(D, D')_t := (D_0, D_0') + t (D_1, D_1')+ t^2 (D_2, D_2') + \cdots \text{ in } \mathrm{Coder}_c^{1}(SV)[[t]]
\end{align}
(with $(D_0, D_0') = (D,D')$) that satisfies $\llbracket (D,D')_t, (D,D')_t \rrbracket = 0$.
\end{defn}

Thus, (\ref{for-def}) defines a formal deformation of $L$ if and only if 
\begin{align}\label{def-eqn}
\sum_{i+j = n} \llbracket (D_i, D_i'), (D_j, D_j') \rrbracket = 0, ~ \text{ for } n \geq 0.
\end{align}
This system of equations are called deformation equations. Note that (\ref{def-eqn}) holds for $n=0$ as $(D,D')$ corresponds to the given compatible $L_\infty$-algebra structure on $L$. However, for $n =1$, we get
\begin{align*}
\llbracket (D,D'), (D_1, D_1') \rrbracket = 0,
\end{align*}
which implies that $(D_1,D'_1) \in C^2_c (L,L)$ is a $2$-cocycle in the cohomology complex of $L$. This is called the infinitesimal of the formal deformation $(D,D')_t.$ In particular, if $(D_1, D_1') = \cdots = (D_{p-1}, D_{p-1}') = 0$ and $(D_p, D_p')$ is the first nonzero term, then $(D_p, D_p') \in C^2_c(L,L)$ is a $2$-cocycle.

\begin{defn}
Let $(D, D')_t = \sum_{i=0}^\infty t^i (D_i, D_i')$ and $\overline{(D,D')}_t = \sum_{i=0}^\infty t^i (\overline{D}_i, \overline{D}_i')$ be two formal deformations of the compatible $L_\infty$-algebra $L$. They are said to be equivalent if there is a formal power series
\begin{align*}
\Phi_t := \Phi_0 + t \Phi_1 + t^2 \Phi_2 + \cdots \text{ in } \mathrm{Coder}^0_c (SV)[[t]] = \mathrm{Coder}^0(SV) [[t]]
\end{align*}
(with $\Phi_0 = \mathrm{id}$) satisfying $\overline{(D, D')}_t \circ \Phi_t = \Phi_t \circ (D,D')_t.$
\end{defn}

It follows that $(D, D')_t$ and $\overline{(D, D')}_t$ are equivalent if and only if
\begin{align}\label{equiv-e}
\sum_{i+j = n}  (\overline{D}_i, \overline{D}_i') \circ \Phi_j = \sum_{i+j = n} \Phi_i \circ (D_j, D_j') , \text{ for } n \geq 0.
\end{align}
Observe that (\ref{equiv-e}) holds automatically for $n=0$ as $\Phi_0 = \mathrm{id}$. For $n = 1$, we get $(D_1, D_1') - (D_1, D_1') = \delta_c (\Phi_1)$. This shows that the infinitesimals corresponding to equivalent deformations are cohomologous. Hence, they gives rise to the same cohomology class in $H^2_c (L,L)$.

Next, we consider rigidity of compatible $L_\infty$-algebras.
\begin{defn}
(i) A deformation $(D,D')_t$ of a compatible $L_\infty$-algebra $L$ is said to be trivial if it is equivalent to the undeformed one $(D,D')$.

(ii) A compatible $L_\infty$-algebra $L$ is said to be rigid if any deformation of $L$ is trivial.
\end{defn}

The following result provides a sufficient condition for the rigidity of a compatible $L_\infty$-algebra.

\begin{thm}\label{2-rigid}
Let $L$ be a compatible $L_\infty$-algebra. If $H^2_c (L,L) = 0$ then $L$ is rigid.
\end{thm}

\begin{proof}
Let $(D, D')_t$ be any deformation of $L$ given by
\begin{align*}
(D,D')_t = (D, D') + \sum_{i \geq p} t^i (D_i, D_i'), ~ \text{ for some } p \geq 1.
\end{align*}
Then we know that the first nonzero term $(D_p, D_p')$ is a $2$-cocycle. By the hypothesis, we have $(D_p, D_p') = \delta_c (\Phi_p)$, for some $\Phi_p \in C^1_c (L,L) = \mathrm{Coder}^0 (SV)$. We consider $\Phi_t = \mathrm{id} + t^p \Phi_p \in \mathrm{Coder}^0 (SV)[[t]]$, and define
$\overline{(D, D')}_t = \Phi_t^{-1} \circ (D, D')_t \circ \Phi_t.$
Then $\overline{(D, D')}_t$ is a deformation equivalent to $(D,D')_t$, and of the form
\begin{align*}
\overline{(D, D')}_t = (D, D') + \sum_{i \geq p+1} t^i (\overline{D}_i, \overline{D}_i').
\end{align*}
By repeating the same process, we can show that $(D, D')_t$ is equivalent to $(D,D')$. Hence the result follows.
\end{proof}

\medskip

\noindent {\bf Finite order deformations.} Here we consider finite order deformations of a compatible $L_\infty$-algebra and their extensions. Let $L=(L, \{ l_k \}_{k \geq 1}, \{ l_k' \}_{k \geq 1})$ be a compatible $L_\infty$-algebra given by a Maurer-Cartan elements $(D,D') \in \mathrm{Coder}^{1}_c (SV)$ in the graded Lie algebra $(\mathrm{Coder}^\bullet_c (SV), \llbracket ~, ~ \rrbracket)$.

\begin{defn}
Let $N \in \mathbb{N}$ be a natural number. An order $N$ deformation of $L$ consists of a finite polynomial
\begin{align*}
(D,D')_t = (D_0,D_0') + t (D_1, D_1') + \cdots + t^N (D_N, D_N') \text{ in } \text{ Coder}^{1}_c (SV)[[t]]/(t^{N+1})
\end{align*}
(with $(D_0,D_0') = (D,D')$) that satisfies $\llbracket (D,D')_t, (D,D')_t \rrbracket = 0$ (\text{mod }$t^{N+1}$). 
\end{defn}

Thus, in an order $N$ deformation, the equations in (\ref{def-eqn}) hold for $n=0,1,\ldots,N$. These equations can be equivalently written as
\begin{align}\label{fin-def-eq}
\llbracket (D,D'), (D_n, D_n') \rrbracket = - \frac{1}{2} \sum_{i+j = n; i, j \geq 1} \llbracket (D_i, D_i'), (D_j, D_j') \rrbracket, ~ \text{ for } n = 0,1, \ldots, N.
\end{align}

\begin{defn}
An order $N$ deformation $(D,D')_t = \sum_{i=0}^N t^i (D_i, D_i')$ is said to be extensible if there exists an element $(D_{N+1}, D_{N+1}') \in \mathrm{Coder}^{1}_c (SV)$ such that $\widetilde{(D,D')}_t = (D,D')_t + t^{N+1} (D_{N+1}, D_{N+1}') = \sum_{i=0}^{N+1} t^i (D_i, D_i')$ is an order $N+1$ deformation.
\end{defn}

Thus, together with equations in (\ref{fin-def-eq}), one more equation need to be satisfied, namely,
\begin{align}\label{fin-def-nn}
\llbracket (D,D'), (D_{N+1}, D_{N+1}') \rrbracket = - \frac{1}{2} \sum_{i+j = N+1; i, j \geq 1} \llbracket (D_i, D_i'), (D_j, D_j') \rrbracket.
\end{align}
The right hand side of the above equation does not contain the term $(D_{N+1}, D_{N+1}')$. Therefore, it depends only on the given orden $N$ deformation $(D,D')_t$. It is called the obstruction to extend the deformation $(D,D')_t$.

\begin{prop}
The obstruction is a $3$-cocycle in the cohomology complex of $L$.
\end{prop}

\begin{proof}
We have
\begin{align*}
&\delta_c \big( - \frac{1}{2} \sum_{i+j = N+1; i, j \geq 1} \llbracket (D_i, D_i'), (D_j, D_j') \rrbracket  \big) \\
&= - \frac{1}{2} \sum_{i+j = N+1; i, j \geq 1}  \llbracket (D,D'), \llbracket   (D_{i}, D'_{i}), (D_{j}, D'_{j}) \rrbracket  \rrbracket \\
&= - \frac{1}{2} \sum_{i+j = N+1;i, j \geq 1}  \big(      \llbracket \llbracket  (D, D'), (D_{i}, D'_{i}) \rrbracket, (D_{j}, D'_{j}) \rrbracket  ~-~ \llbracket (D_{i}, D'_{i}), \llbracket (D, D'), (D_{j}, D'_{j}) \rrbracket \rrbracket \big) \\
&= \frac{1}{4} \sum_{i_1 + i_2 + j = N+1;i_1, i_2, j \geq 1} \llbracket \llbracket (D_{i_1}, D'_{i_1}), (D_{i_2}, D'_{i_2}) \rrbracket, (D_{j}, D'_{j}) \rrbracket  \\
& \qquad - \frac{1}{4} \sum_{i+j_1 + j_2 = N+1; i, j_1, j_2 \geq 1} \llbracket (D_{i}, D'_{i}), \llbracket (D_{j_1}, D'_{j_1}), (D_{j_2}, D'_{j_2}) \rrbracket \rrbracket    \qquad (\text{form } (\ref{fin-def-eq})) \\
&= \frac{1}{2} \sum_{i+j+k = N+1; i, j, k \geq 1} \llbracket \llbracket (D_{i}, D'_{i}), (D_{j}, D'_{j}) \rrbracket , (D_{k}, D'_{k}) \rrbracket  = 0.
\end{align*}
Hence the proof.
\end{proof}

Therefore, it follows that the obstruction induces a cohomology class in $H^3_c (L,L)$. This is called the obstruction class. As a consequence of (\ref{fin-def-nn}), we get the following.

\begin{thm}\label{3-ext}
A finite order deformation of a compatible $L_\infty$-algebra $L$ is extensible if and only if the corresponding class in $H^3_c (L,L)$ is trivial.
\end{thm}

\begin{corollary}
Let $L$ be a compatible $L_\infty$-algebra. If $H^3_c (L,L) = 0$ then any finite order deformation of $L$ is extensible.
\end{corollary}

\section{Classification of some compatible $L_\infty$-algebras}\label{sec6}

In this section, we will consider compatible $L_\infty$-algebras whose underlying graded vector space is concentrated in degrees $-1$ and $0$. We call them $2$-term compatible $L_\infty$-algebras. We first recall $2$-term $L_\infty$-algebras \cite{baez-crans}.

\begin{defn}
(i) A $2$-term $L_\infty$-algebra is a tuple $(L_{-1} \oplus L_0, l_1, l_2, l_3)$ consisting of a graded vector space $L_{-1} \oplus L_0$ (concentrated in degrees $-1$ and $0$) together with a linear map $l_1 : L_{-1} \rightarrow L_0$, skew-symmetric bilinear maps $l_2 : L_i \times L_j \rightarrow L_{i+j}$ ($-1 \leq i, j , i+j \leq 0$) and a skew-symmetric trilinear map $l_3 : L_0 \times L_0 \times L_0 \rightarrow L_{-1}$ satisfying for $w, x, y, z \in L_0$ and $h, k \in L_{-1},$
\begin{align}
& l_1 (l_2 (x, h)) = l_2 (x, l_1 (h)),\\
& l_2 (l_1 (h), k) = l_2 (h, l_1(k)),\\
& l_1 (l_3 (x,y,z)) = - l_2 ( l_2(x,y), z) + l_2 (l_2 (x,z), y) + l_2 (x, l_2 (y,z)), \label{sk-lie}\\
& l_3 (l_1 (h), x, y) = - l_2 ( l_2(x,y), h) + l_2 (l_2 (x,h), y) + l_2 (x, l_2 (y,h)), \label{sk2-lie}\\
&l_2 (l_3 (w,x, y), z) + l_2 ( l_3 (w,y,z), x) - l_2 (l_3 (w,x,z), y) - l_2 (l_3 (x,y,z), w) 
= l_3 (l_2 (w,x ),y,z ) \label{sk3-lie} \\&- l_3 (l_2 (w,y), x, z) + l_3 (l_2 (w,z), x, y) + l_3 (l_2 (x, y), w, z) - l_3 (l_2 (x,z), w, y) + l_3 (l_2 (y,z), w, x). \nonumber
\end{align}

(ii) Let $(L_{-1} \oplus L_0, l_1, l_2, l_3)$ and $(M_{-1} \oplus M_0, m_1, m_2, m_3)$  be $2$-term $L_\infty$-algebras. A morphism between them is a triple $(f_{-1}, f_0, f)$ consisting of linear maps $f_{-1} : L_{-1} \rightarrow M_{-1}$ and $f_0 : L_0 \rightarrow M_0$, and a skew-symmetric bilinear map $f : L_0 \times L_0 \rightarrow M_{-1}$ satisfying for $x,y,z \in L_0$ and $h \in L_{-1}$,
\begin{align}
&f_0 \circ l_1 = m_1 \circ f_{-1},\\
& m_1 (f (x,y)) = f_0 (l_2 (x,y)) - m_2 (f_0(x), f_0(y)),\\
& f (x, l_1 (h)) = f_{-1} (l_2 (x, h)) - m_2 (f_0 (x), f_{-1}(h)),\\
& m_2 (f(x,y), f_0 (z)) + f (l_2(x,y), z) +  f_{-1} (l_3 (x,y,z)) = m_3 (f_0 (x), f_0 (y), f_0 (z)) \\
&+m_2 (f_0 (x), f (y,z)) + m_2 (f(x,z), f_0 (y)) + f (x, l_2 (y, z)) + f (l_2 (x,z), y)). \nonumber
\end{align}
\end{defn}


\begin{defn}
A $2$-term compatible $L_\infty$-algebra is a tuple  $(L_{-1} \oplus L_0, \{ l_1, l_2, l_3 \}, \{ l'_1 , l'_2, l'_3 \})$ in which $(L_{-1} \oplus L_0, l_1, l_2, l_3)$ and $(L_{-1} \oplus L_0, l'_1, l'_2, l'_3)$ are $2$-term $L_\infty$-algebras satisfying for $w, x, y, z \in L_0$ and $h, k \in L_{-1},$
\begin{align}
& l_1 (l'_2 (x, h)) + l'_1 (l_2 (x, h)) = l_2 (x, l'_1 (h)) + l'_2 (x, l_1 (h)), \label{comp-2-first}\\
& l_2 (l'_1 (h), k) + l'_2 (l_1 (h), k) = l_2 (h, l'_1(k)) + l'_2 (h, l_1(k)),\\
& l_1 (l'_3 (x,y,z)) + l'_1 (l_3 (x,y,z)) =- l_2 ( l'_2(x,y), z) - l'_2 ( l_2(x,y), z) + l_2 (l'_2 (x,z), y) + l'_2 (l_2 (x,z), y) \label{comp-2-id}\\
&\qquad \qquad \qquad \qquad \qquad \qquad \qquad + l_2 (x, l'_2 (y,z)) + l'_2 (x, l_2 (y,z)), \nonumber \\
& l_3 (l'_1 (h), x, y) + l'_3 (l_1 (h), x, y) = - l_2 ( l'_2(x,y), h) - l'_2 ( l_2(x,y), h) + l_2 (l'_2 (x,h), y) + l'_2 (l_2 (x,h), y) \label{comp-2-idd}\\
& \qquad \qquad \qquad \qquad \qquad \qquad \qquad + l_2 (x, l'_2 (y,h)) + l'_2 (x, l_2 (y,h)), \nonumber \\
& l_2 (l'_3 (w,x, y), z) + l'_2 (l_3 (w,x, y), z) + l_2 ( l'_3 (w,y,z), x) + l'_2 ( l_3 (w,y,z), x) - l_2 (l'_3 (w,x,z), y) \label{comp-2-iddd}\\
&- l'_2 (l_3 (w,x,z), y) - l_2 (l'_3 (x,y,z), w) - l'_2 (l_3 (x,y,z), w) = l_3 (l'_2 (w,x ),y,z ) + l'_3 (l_2 (w,x ),y,z ) \nonumber \\
&- l_3 (l'_2 (w,y), x, z) - l'_3 (l_2 (w,y), x, z) + l_3 (l'_2 (w,z), x, y) + l'_3 (l_2 (w,z), x, y) + l_3 (l'_2 (x, y), w, z) \nonumber \\
&+ l'_3 (l_2 (x, y), w, z) - l_3 (l'_2 (x,z), w, y) - l'_3 (l_2 (x,z), w, y) + l_3 (l'_2 (y,z), w, x) + l'_3 (l_2 (y,z), w, x). \nonumber
\end{align}
\end{defn}

\begin{remark}
Let $(L_{-1} \oplus L_0, \{ l_1, l_2, l_3 \}, \{ l'_1 , l'_2, l'_3 \})$ be a $2$-term compatible $L_\infty$-algebra with $l_1 = l'_1 = l.$ Then the compatibility conditions (\ref{comp-2-first}) - (\ref{comp-2-iddd}) is equivalent to the fact that $(L_{-1} \oplus L_0, l, l_2 + l'_2, 2(l_3+ l'_3))$ is a $2$-term $L_\infty$-algebra.
\end{remark}

\begin{defn}
Let $(L_{-1} \oplus L_0, \{ l_1, l_2, l_3 \}, \{ l'_1 , l'_2, l'_3 \})$ and $(M_{-1} \oplus M_0, \{ m_1, m_2, m_3 \}, \{ m'_1 , m'_2, m'_3 \})$ be $2$-term compatible $L_\infty$-algebras. A morphism between them is a triple $(f_{-1}, f_0, f)$ which is an $L_\infty$-algebra morphism from $(L_{-1} \oplus L_0, l_1, l_2, l_3 )$ to $(M_{-1} \oplus M_0,  m_1, m_2, m_3)$ and an $L_\infty$-algebra morphism from $(L_{-1} \oplus L_0, l'_1, l'_2, l'_3 )$ to $(M_{-1} \oplus M_0,  m'_1, m'_2, m'_3)$.
\end{defn}

Let $(f_{-1}, f_0, f) : (L_{-1} \oplus L_0, \{ l_1, l_2, l_3 \}, \{ l'_1 , l'_2, l'_3 \}) \rightsquigarrow (M_{-1} \oplus M_0, \{ m_1, m_2, m_3 \}, \{ m'_1 , m'_2, m'_3 \})$ and $(g_{-1}, g_0, g) : (M_{-1} \oplus M_0, \{ m_1, m_2, m_3 \}, \{ m'_1 , m'_2, m'_3 \}) \rightsquigarrow (N_{-1} \oplus N_0, \{ n_1, n_2, n_3 \}, \{ n'_1 , n'_2, n'_3 \})$ be two morphisms of compatible $L_\infty$-algebras. Their composition is defined by the triple $(g_{-1} \circ f_{-1},~ g_0 \circ f_0,~ g \circ (f_0 \times f_0) + g_{-1} \circ f)$.

\begin{notation}\label{notat}
We denote by ${\bf 2CompL_\infty}$ the category of $2$-term compatible $L_\infty$-algebras and morphisms between them. Let ${\bf 2CompL^1_\infty}$ be the subcategory of ${\bf 2CompL_\infty}$ whose objects are $2$-term compatible $L_\infty$-algebras of the form $(L_{-1} \oplus L_0, \{ l_1, l_2, l_3 \}, \{ l'_1 , l'_2, l'_3 \})$ with $l_1=l'_1.$
\end{notation}

The following definitions are motivated by the standard $L_\infty$ case \cite{baez-crans}.

\begin{defn}
A $2$-term compatible $L_\infty$-algebra  $(L_{-1} \oplus L_0, \{ l_1, l_2, l_3 \}, \{ l'_1 , l'_2, l'_3 \})$ is said to be
\begin{itemize}
\item[(i)] `strict' if $l_3=l'_3 = 0$,
\item[(ii)] `skeletal' if $l_1 = l'_1 = 0$.
\end{itemize}
\end{defn}

In the rest of this section, we will be interested in strict and skeletal compatible $L_\infty$-algebras and their classifications.

\subsection*{Crossed modules of compatible Lie algebras and strict compatible $L_\infty$-algebras}

In this subsection, we introduce crossed module of compatible Lie algebras and show that they are closely related to strict compatible $L_\infty$-algebras.

\begin{defn}
A crossed module of compatible Lie algebras is tuple $(\mathfrak{g}, \mathfrak{h}, t, t', \alpha, \alpha')$ consisting of compatible Lie algebras $\mathfrak{g} = (\mathfrak{g}, [~,~], [~,~]')$ and $\mathfrak{h} = (\mathfrak{h}, [~,~]_\mathfrak{h}, [~,~]'_\mathfrak{h})$, homomorphisms $t : (\mathfrak{g}, [~,~]) \rightarrow (\mathfrak{h}, [~,~]_\mathfrak{h})$ and $t' : (\mathfrak{g}, [~,~]') \rightarrow (\mathfrak{h}, [~,~]'_\mathfrak{h})$ of Lie algebras, and bilinear maps $ \alpha, \alpha' : \mathfrak{h} \times \mathfrak{g} \rightarrow \mathfrak{g}$ satisfying for $m, n \in \mathfrak{g}$ and $x, y \in \mathfrak{h}$,
\begin{itemize}
\item[(i)] $t (\alpha (x, m)) = [x, t (m)]_\mathfrak{h}$,\\
$t' (\alpha' (x, m)) = [x, t' (m)]'_\mathfrak{h}$,\\
$ t (\alpha' (x,m) ) + t' ( \alpha (x,m)) = [x, t'(m)]_\mathfrak{h} + [ x, t (m)]'_\mathfrak{h},$
\item[(ii)] $\alpha (t(m), n) = [m, n]$,\\
$\alpha' (t'(m), n) = [m, n]'$,\\
$\alpha (t'(m), n) + \alpha' (t(m), n) = \alpha (m, t'(n)) + \alpha' (m, t(n))$,
\item[(iii)] $\alpha (x, [m,n]) = [\alpha (x, m), n] + [m, \alpha (x,n)],$\\
$\alpha' (x, [m,n]') = [\alpha' (x, m), n]' + [m, \alpha' (x,n)]',$\\
$\alpha (x, [m,n]') + \alpha' (x, [m,n]) = [\alpha (x, m), n]' + [\alpha' (x, m), n] + [m, \alpha (x,n)]' + [m, \alpha' (x,n)],$
\item[(iv)] $\alpha ([x,y], m) = \alpha (x, \alpha (y, m)) - \alpha (y, \alpha (x, m)),$ \\
$\alpha' ([x,y]', m) = \alpha' (x, \alpha' (y, m)) - \alpha' (y, \alpha' (x, m)),$\\
$\alpha ([x,y]', m) + \alpha' ([x,y], m) = \alpha (x, \alpha' (y, m)) + \alpha' (x, \alpha (y, m)) - \alpha (y, \alpha' (x, m)) - \alpha' (y, \alpha (x, m))$.
\end{itemize}
\end{defn}

\begin{thm}\label{strict-thm}
There is a one-to-one correspondence between strict compatible $L_\infty$-algebras and crossed module of compatible Lie algebras.
\end{thm}

\begin{proof}
Let $(L_{-1} \oplus L_0, \{ l_1, l_2, 0 \}, \{ l'_1, l'_2, 0 \})$ be a strict compatible $L_\infty$-algebra. Since $(L_{-1} \oplus L_0,  l_1, l_2, 0 )$ is a strict $L_\infty$-algebra \cite{baez-crans}, it follows that $\mathfrak{g} = L_{-1}$ and $\mathfrak{h} = L_0$ both carries Lie brackets given by 
\begin{align*}
[h, k ]:= l_2 ( l_1 (h), k ) = l_2 (h, l_1 (k)), ~ \text{ for } h, k \in L_{-1} ~~~ \text{ and } ~~~ [x, y]_\mathfrak{h} := l_2 (x,y),~\text{ for } x, y \in L_0.
\end{align*}
Similarly,  $(L_{-1} \oplus L_0,  l'_1, l'_2, 0 )$ is a strict $L_\infty$-algebra implies that $\mathfrak{g} = L_{-1}$ and $\mathfrak{h} = L_0$ equipped with another Lie brackets
\begin{align*}
[h,k]' := l_2' (l'_1 (h), k) = l'_2 (h, l'_1 (k)), \text{ for } h, k \in L_{-1} ~~~ \text{ and } ~~~ [x,y]'_\mathfrak{h} := l'_2 (x,y),~ \text{ for } x, y \in L_0.
\end{align*}
The compatibility conditions (\ref{comp-2-id}) and (\ref{comp-2-idd}) implies that $(\mathfrak{g} = L_{-1}, [~,~], [~,~]')$ and $(\mathfrak{h}= L_0, [~,~]_\mathfrak{h}, [~,~]'_\mathfrak{h})$ are compatible Lie algebras. Finally, we define maps
\begin{align*}
t = l_1 : \mathfrak{g} \rightarrow \mathfrak{h}, ~~~~ t' = l'_1 : \mathfrak{g} \rightarrow \mathfrak{h}, ~~~~ \alpha = l_2 : \mathfrak{h} \times \mathfrak{g} \rightarrow \mathfrak{g} ~~~ \text{ and } ~~~ \alpha' = l'_2 : \mathfrak{h} \times \mathfrak{g} \rightarrow \mathfrak{g}.
\end{align*}
Then it is easy to verify that $(\mathfrak{g}, \mathfrak{h}, t, t', \alpha, \alpha')$ is a crossed module of compatible Lie algebras.

Conversely, let $(\mathfrak{g}, \mathfrak{h}, t, t', \alpha, \alpha')$  be a crossed module of compatible Lie algebras. Then it is easy to verify that $(\underbrace{\mathfrak{g}}_{-1} \oplus \underbrace{\mathfrak{h}}_0, \{ l_1, l_2, 0 \}, \{ l'_1, l'_2 , 0 \})$ is a strict compatible $L_\infty$-algebra, where
\begin{align*}
&l_1 = t, ~~~~ l_2 (x, y) = [x, y]_\mathfrak{h}, ~~~~ l_2 (x,h) = -l_2 (h, x) := \alpha (x, h), ~ \text{ for } x, y \in \mathfrak{h},~ h \in \mathfrak{g},   \\
&l'_1 = t', ~~~~ l'_2 (x, y) = [x, y]'_\mathfrak{h}, ~~~~ l'_2 (x,h) = -l'_2 (h, x) := \alpha' (x, h), ~ \text{ for } x, y \in \mathfrak{h},~ h \in \mathfrak{g}, 
\end{align*}
Finally, the above two constructions are inverses to each other. This completes the proof. 
\end{proof}

\subsection*{Skeletal compatible $L_\infty$-algebras}\label{sec6}

\begin{thm}\label{skeletal-thm}
There is a one-to-one correspondence between skeletal compatible $L_\infty$-algebras and tuples $(\mathfrak{g}, V, (\theta, \theta + \theta' , \theta'))$, where $\mathfrak{g} = (\mathfrak{g}, [~,~], [~,~]')$ is a compatible Lie algebra, $V = (V, \rho, \rho')$ is a representation and $(\theta, \theta + \theta', \theta') \in C^3_c (\mathfrak{g}, V)$ is a $3$-cocycle induced by compatible Lie algebra cocycles.
\end{thm}

\begin{proof}
Let $(L_{-1} \oplus L_0, \{ 0, l_2, l_3 \}, \{ 0, l'_2, l'_3 \})$ be a skeletal compatible $L_\infty$-algebra. Since $(L_{-1} \oplus L_0, 0, l_2, l_3)$ is a skeletal $L_\infty$-algebra \cite{baez-crans}, it follows from (\ref{sk-lie}) that $\mathfrak{g} = L_0$ is a Lie algebra with the bracket $[x,y]:= l_2 (x,y)$, for $x,y \in \mathfrak{g} = L_0$. Moreover, if follows from (\ref{sk2-lie}) that $V = L_{-1}$ is a representation of the Lie algebra $(\mathfrak{g}, [~,~])$ with the action map $\rho : \mathfrak{g} \rightarrow \mathrm{End}(V)$, $\rho (x) (h) := l_2 (x, h)$, for $x \in \mathfrak{g}$ and $h \in V$. Finally, the condition (\ref{sk3-lie}) says that the skew-symmetric trilinear map $l_3 : L_0 \times L_0 \times L_0 \rightarrow L_{-1}$ ~ (i.e. $l_3 \in \mathrm{Hom}(\wedge^3 \mathfrak{g}, V)$) is a $3$-cocycle in the Chevalley-Eilenberg cochain complex of the Lie algebra $(\mathfrak{g}, [~,~])$ with coefficients in the representation $(V, \rho)$, i.e. $\delta (l_3) = 0$. Similarly,  $(L_{-1} \oplus L_0, 0, l'_2, l'_3)$ is a skeletal $L_\infty$-algebra implies that $(\mathfrak{g} = L_0, [~,~]')$ is a Lie algebra, $(V = L_{-1}, \rho')$ is a representation, where
\begin{align*}
[x,y]' := l'_2 (x, y) ~~~~ \text{ and } ~~~~ \rho'(x) (h) = l'_2 (x, h), ~ \text{ for } x, y \in \mathfrak{g} = L_0 \text{ and } h \in V = L_{-1}. 
\end{align*}
And the skew-symmetric trilinear map $l'_3 : L_0 \times L_0 \times L_0 \rightarrow L_{-1}$ (i.e. $l_3 \in \mathrm{Hom}(\wedge^3 \mathfrak{g}, V)$) is a $3$-cocycle in the Chevalley-Eilenberg complex of the Lie algebra $(\mathfrak{g}, [~,~]')$ with coefficients in the representation $(V, \rho')$, i.e. $\delta' (l'_3) = 0$.

Finally, the condition (\ref{comp-2-id}) implies that $(\mathfrak{g}, [~,~], [~,~]')$ is a compatible Lie algebra, the condition (\ref{comp-2-idd}) implies that $(V, \rho, \rho')$ is a representation of the compatible Lie algebra. Lastly, the condition (\ref{comp-2-iddd}) is equivalent to $\delta (l'_3) + \delta' (l_3) = 0$. Therefore, by Remark \ref{rem-comp-co}, we have $\delta_c (l_3, l_3 + l'_3, l'_3) = 0$. As a summary, we get a required triple $(\mathfrak{g}, V, (l_3, l_3 + l'_3 + l'_3))$. 

Conversely, let $(\mathfrak{g}, V, (\theta , \theta + \theta' , \theta'))$ be a triple consisting of a compatible Lie algebra $\mathfrak{g} = (\mathfrak{g}, [~,~], [~,~]')$, a representation $V = (V, \rho, \rho')$ and a $3$-cocycle $(\theta, \theta + \theta', \theta') \in C^3_c (\mathfrak{g}, V)$ induced by compatible Lie algebra cocycles. Then it is easy to verify that $(\underbrace{V}_{-1} \oplus \underbrace{\mathfrak{g}}_{0}, \{ 0, l_2, l_3 \}, \{ 0, l'_2, l'_3 \})$ is a skeletal compatible $L_\infty$-algebra, where
\begin{align*}
&l_2 (x,y) := [x, y], ~~~~ l_2 (x, h) = - l_2 (h, x) = \rho(x) (h), ~~~~ l_3 = \theta,\\
&l'_2 (x,y) := [x, y]', ~~~~ l'_2 (x, h) = - l'_2 (h, x) = \rho'(x) (h), ~~~~ l'_3 = \theta',
\end{align*}
for $x, y \in \mathfrak{g}$ and $h \in V$. The above two correspondences are inverses to each other. Hence the proof.
\end{proof}

\section{Compatible Lie $2$-algebras}\label{sec7}
In this section, we introduce compatible Lie $2$-algebras as the categorification of compatible Lie algebras. We show that the category of compatible Lie $2$-algebras is equivalent to category ${\bf 2CompL^1_\infty}$ introduced in Section \ref{sec6} (Notation \ref{notat}).

Like a (compatible) Lie algebra has an underlying vector space, a (compatible) Lie $2$-algebra has an underlying $2$-vector space. A $2$-vector space is basically a category internal to the category {\bf Vect} of vector spaces. More precisely, a $2$-vector space $\mathcal{C}$ is a category with a vector space of objects $\mathcal{C}_0$ and a vector space of morphisms $\mathcal{C}_1$ such that the source and target maps $s, t : \mathcal{C}_1 \rightarrow \mathcal{C}_0$, the identity assigning map $i : \mathcal{C}_0 \rightarrow \mathcal{C}_1$, and the composition map $\circ : C_1 \times_{\mathcal{C}_0} \mathcal{C}_1$ are all linear maps. Given $2$-vector spaces $\mathcal{C}$ and $\mathcal{D}$, a morphism $F : \mathcal{C} \rightarrow \mathcal{D}$ is just a functor in the category  {\bf Vect}.

\begin{defn}\cite{baez-crans} A Lie $2$-algebra is a $2$-vector space $\mathcal{C}$ together with a skew-symmetric bilinear functor $[~,~]: \mathcal{C} \times \mathcal{C} \rightarrow \mathcal{C}$ and a skew-symmetric trilinear natural isomorphism (called the Jacobiator)
\begin{align*}
J_{x,y,z} : [[x,y], z] \rightarrow [[x,z], y]+[x, [y,z]], \text{ for } x, y, z \in \mathcal{C}_0
\end{align*}
which makes the following diagram commutative: for $w, x, y, z \in \mathcal{C}_0$,
\begin{align*}
\xymatrixrowsep{0.6in}
\xymatrixcolsep{0.1in}
\xymatrix  {
& [[[w, x],y],z] \ar[rd]^{J_{[w,x], y, z}} \ar[ld]_{[J_{w,x,y}, z]} & \\
[[[w,y],x],z] +  [[[w, [x,y]],z] \ar[d]^{J_{[w,y],x,z} + J_{w,[x,y], z}} & & [[[w,x],z],y]+[[w,x],[y,z]] \ar[d]_{[J_{w,x,z}, y] + J_{w,x,[y,z]}} \\ 
{ \substack{  [[[w,y],z],x]+[[w,y],[x,z]] \\ ~~~ +[[w,z],[x,y]] + [w,[[x,y],z]]}}  \ar[rd]_{[J_{w,y,z}, x] + 1 + 1+ [w, J_{x,y,z}]}
  & & { \substack{  [[[w,z],x],y] + [[w, [x,z]],y]  \\ ~~~ + [[w, [y,z]],x] + [w, [x,[y,z]]] }}  \ar[ld]^{J_{[w,z],x,y} + J_{w, [x,z],y} + 1 + 1} \\
 &  { \substack{ [[[w,z],y],x] + [[w,z],[x,y]] + [[w,y],[x,z]] \\ ~~~+ [w, [[x,z], y]] + [[w,[y,z]], x] + [w, [x,[y,z]]]. }} & 
}
\end{align*}
A Lie $2$-algebra as above may be denoted by the triple $(\mathcal{C}, [~,~], J)$.
\end{defn}

\begin{defn}
Let $(\mathcal{C}, [~,~], J)$ and $(\mathcal{D}, \lfloor ~, ~\rfloor , K)$ be Lie $2$-algebras. A morphism between them consists of a pair $(F, G)$ of a linear functor $F : \mathcal{C} \rightarrow \mathcal{D}$ on the underlying $2$-vector spaces and a skew-symmetric bilinear natural transformation $G_{x, y} : \lfloor F(x), F(y)\rfloor \rightarrow F[x,y]$, for $x, y \in \mathcal{C}_0$, satisfying
\begin{align*}
\xymatrix  {
\lfloor \lfloor F(x), F(y) \rfloor, F(z) \rfloor \ar[r]^{K_{F(x), F(y), F(z)}} \ar[d]_{\lfloor G_{x,y}, 1 \rfloor} & \lfloor \lfloor F(x), F(z) \rfloor, F(y) \rfloor + \lfloor F(x), \lfloor F(y), F(z) \rfloor \rfloor \ar[d]^{   \lfloor G_{x,z}, 1 \rfloor + \lfloor 1, G_{y,z} \rfloor} \\
\lfloor F[x,y], F(z) \rfloor \ar[d]_{G_{[x,y], z}} & \lfloor F[x,z], F(y) \rfloor + \lfloor F(x), F[y,z] \rfloor \ar[d]^{ G_{[x,z],y} + G_{x, [y,z]}  } \\
F[[x,y], z] \ar[r]_{F(J_{x,y,z})} &  F [[x,z],y] + F [x, [y,z]].
}
\end{align*}
\end{defn}

In the following, we introduce compatible Lie $2$-algebras and morphism between them. 

\begin{defn}
A compatible Lie $2$-algebra is a tuple $(\mathcal{C}, [~,~], J, [~,~]', J', J^c)$ in which $(\mathcal{C}, [~,~], J)$ and $(\mathcal{C}, [~,~]', J')$ are Lie $2$-algebras and $J^c$ is a trilinear skew-symmetric natural isomorphism
\begin{align*}
J^c_{x, y, z} : [[x,y], z]' + [[x,y]', z] \rightarrow [[x,z], y]' + [[x,z]', y] + [x, [y,z] ]' + [x, [y,z]' ]
\end{align*}
which makes the triple $(\mathcal{C}, [~,~] + [~,~]', J + J' + J^c)$ into a Lie $2$-algebra.
\end{defn} 

\begin{defn}
Let $(\mathcal{C}, [~,~], J, [~,~]', J', J^c)$ and $(\mathcal{D}, \lfloor~,~ \rfloor, K, \lfloor ~, ~\rfloor', K', K^c)$ be two compatible Lie $2$-algebras. A morphism between them is a triple $(F, G, G')$ consisting of a linear funtor $F : \mathcal{C} \rightarrow \mathcal{D}$ on the underlying $2$-vector spaces, and two skew-symmetric bilinear natural transformations
\begin{align*}
G_{x,y}: \lfloor F(x), F(y) \rfloor \rightarrow F[x,y]    ~~~~ \text{ and } ~~~ G'_{x,y}: \lfloor F(x), F(y) \rfloor' \rightarrow F[x,y]'
\end{align*}
such that $(F,G)$ is a morphism of Lie $2$-algebras from $(\mathcal{C}, [~,~], J)$ to $(\mathcal{D}, \lfloor~,~\rfloor, K)$, and $(F,G')$ is a morphism of Lie $2$-algebras from $(\mathcal{C}, [~,~]', J')$ to $(\mathcal{D}, \lfloor~,~\rfloor', K').$
\end{defn}

Compatible Lie $2$-algebras and morphisms between them forms a category. We denote this category by {\bf CompLie2}.

\begin{thm}\label{2-2thm}
The categories {\bf 2CompL$^1_\infty$} and  {\bf CompLie2} are equivalent.
\end{thm}

\begin{proof}
Let $L = (L_{-1} \oplus L_0, \{ l, l_2, l_3 \}, \{ l, l'_2, l'_3 \})$ be an element in {\bf 2CompL$^1_\infty$}. We define a $2$-vector space $\mathcal{C}$ whose set of objects $\mathcal{C}_0 = L_0$ and the set of morphisms $\mathcal{C}_1 = L_{-1} \oplus L_0$, and structure maps
\begin{align*}
s (h, x) = x, ~~~~ t (h, x) = l(h) + x  ~~\text{ and } ~~ i(x) = (0, x).
\end{align*}
Define a skew-symmetric bilinear functor $[~,~] : \mathcal{C} \times \mathcal{C} \rightarrow \mathcal{C}$ and a skew-symmetric trilinear natural isomorphism $J$ by
\begin{align*}
[(h,x), (k, y)] := ( l_2 (h, y) + l_2 (x,k) + l_2 (dh, k),~ l_2 (x,y)) ~~ \text{ and } ~~
J_{x,y,z} := \big(  l_3 (x,y,z), ~l_2 (l_2 (x,y), z)  \big).
\end{align*}
Since $(L_{-1} \oplus L_0, l, l_2, l_3)$ is a $2$-term $L_\infty$-algebra, it follows from \cite{baez-crans} that $(\mathcal{C}, [~,~], J)$ is a Lie $2$-algebra. Similarly, $(L_{-1} \oplus L_0, l, l'_2, l'_3)$ is a $2$-term $L_\infty$-algebra implies that $(\mathcal{C}, [~,~]', J')$ is a Lie $2$-algebra, where
\begin{align*}
[(h,x), (k, y)]' := ( l'_2 (h, y) + l'_2 (x,k) + l'_2 (dh, k),~ l'_2 (x,y)) ~~ \text{ and } ~~
J'_{x,y,z} := \big(  l'_3 (x,y,z), ~l'_2 (l'_2 (x,y), z)  \big).
\end{align*}
Finally, we define another skew-symmetric trilinear natural isomorphism $J^c$ by
\begin{align*}
J^c_{x,y,z} := (l_3 (x,y,z) + l'_3 (x,y,z),~ l_2 (l'_2 (x,y), z) + l'_2 (l_2 (x,y), z)).
\end{align*}
By various conditions (\ref{comp-2-first})-(\ref{comp-2-iddd}) of the compatibility of the $L_\infty$-algebras $(L_{-1} \oplus L_0, l, l_2, l_3)$ and $(L_{-1} \oplus L_0, l, l'_2, l'_3)$, we deduce that $(\mathcal{C}, [~,~], J, [~,~]', J', J^c)$ is a compatible Lie $2$-algebra. We denote this compatible Lie $2$-algebra by $\Phi (L).$

Next, let $(f_{-1}, f_0, f) : \underbrace{(L_{-1} \oplus L_0, \{ l, l_2, l_3 \}, \{ l, l'_2, l'_3 \})}_L \rightsquigarrow \underbrace{(M_{-1} \oplus M_0, \{ m, m_2, m_3 \}, \{ m, m'_2, m'_3 \})}_M$ be a morphism of $2$-term compatible $L_\infty$-algebras. We define a morphism $(F,G,G'): \Phi (L) \rightsquigarrow \Phi (M) $ of compatible Lie $2$-algebras by
\begin{align*}
F ((h,x)) = (f_{-1}(h), f_0 (x)), ~~ G_{x,y} = (f (x,y),~ l_2 (f_0 (x), f_0 (y) )) ~~ \text{ and } ~~ G'_{x,y} = (f (x,y),~ l'_2 (f_0 (x), f_0 (y) )).
\end{align*}One can also show that such construction preserves identity morphisms and the composition of morphisms. Thus, we obtain a functor $\Phi : {\bf 2CompL^1_\infty} \rightarrow {\bf CompLie2}$.

\medskip

Conversely, let $\mathcal{C} = (\mathcal{C}, [~,~], J, [~,~]', J', J^c)$ be a compatible Lie $2$-algebra. We define two vector spaces $L_{-1}, L_0$ and a linear map $l : L_{-1} \rightarrow L_0$ by
\begin{align*}
L_{-1} = \mathrm{ker}(s), ~~~ L_0 = \mathcal{C}_0 ~~~ \text{ and } ~~~ l = t|_{\mathrm{ker}(s)}.
\end{align*}
Define a skew-symmetric bilinear map $l_2 : L_i \times L_j \rightarrow L_{i+j}$ $(-1 \leq i, j , i+j \leq 0)$ and a skew-symmetric trilinear map $l_3 : L_0 \times L_0 \times L_0 \rightarrow L_{-1}$ by
\begin{align*}
l_2 (x, y) = [x, y], ~~~ l_2 (x, h) = -l_2 (h, x) = [i(x), h] ~~ \text{ and } ~~ l_3 (x, y,z) = \mathrm{pr}_1 (J_{x, y, z}).
\end{align*}
Since $(\mathcal{C}, [~,~], J)$ is a Lie $2$-algebra, it follows from \cite{baez-crans} that $(L_{-1} \oplus L_0, l, l_2, l_3)$ is a $2$-term $L_\infty$-algebra. Similarly, $(\mathcal{C}, [~,~]', J')$ is a Lie $2$-algebra implies that $(L_{-1} \oplus L_0, l, l'_2, l'_3)$ is a  $2$-term $L_\infty$-algebra, where
\begin{align*}
l'_2 (x, y) = [x, y]', ~~~ l'_2 (x, h) = -l'_2 (h, x) = [i(x), h]' ~~ \text{ and } ~~ l'_3 (x, y,z) = \mathrm{pr}_1 (J'_{x, y, z}).
\end{align*}
Finally, the compatibility of Lie $2$-algebras $(\mathcal{C}, [~,~], J)$ and $(\mathcal{C}, [~,~]', J')$ implies that $(L_{-1} \oplus L_0, \{ l, l_2, l_3 \}, \{ l, l'_2, l'_3 \})$ is a $2$-term compatible $L_\infty$-algebra. We denote this algebra by $\Psi (\mathcal{C})$.

Next, let $(F, G, G')$ be a morphism from $\underbrace{(\mathcal{C}, [~,~], J, [~,~]', J', J^c)}_\mathcal{C}$ to $\underbrace{(\mathcal{D}, \lfloor ~,~ \rfloor, K, \lfloor ~,~ \rfloor', K', K^c)}_\mathcal{D}$. Then it follows from \cite{baez-crans} that $(f_{-1}, f_0, f) : \Psi (\mathcal{C}) \rightsquigarrow \Psi (\mathcal{D})$ is a morphism of compatible $L_\infty$-algebras, where
\begin{align*}
f_{-1} := F|_{\mathrm{ker}(s)}, ~~~ f_0 := F|_{\mathcal{C}_0} ~~ \text{ and } ~~ f(x,y) := G(x,y) - i ( s (G(x,y)) )= G'(x,y) - i ( s (G'(x,y)) ). 
\end{align*} 
This construction preserves the identity morphisms and the composition of morphisms. Hence we get a functor $\Psi : {\bf CompLie2} \rightarrow {\bf 2CompL^1_\infty}.$

\medskip

We are now left to show that there are natural isomorphisms $\alpha : \Psi \circ \Phi \Rightarrow 1_{ {\bf 2CompL^1_\infty} }$ and $\beta : \Phi \circ \Psi \Rightarrow 1_{\bf CompLie2}$.  Let $L = (L_{-1} \oplus L_0, \{ l, l_2, l_3 \}, \{ l, l'_2, l'_3 \})$ be a compatible $L_\infty$-algebra in ${\bf 2CompL^1_\infty}$. By applying $\Psi \circ \Phi$, we get that $(\Psi \circ \Phi)(L)$ is same as $L$. Hence $\alpha$ is the identity transformation.

On the other hand, let $\mathcal{C} = (\mathcal{C}, [~,~], J, [~,~]', J', J^c)$ be a compatible Lie $2$-algebra. If we apply $\Psi$, we get the $2$-term compatible $L_\infty$-algebra
\begin{align*}
\Psi (\mathcal{C}) = (\mathrm{ker}(s) \oplus \mathcal{C}_0, \{ l = t|_{\mathrm{ker}(s)} , l_2, l_3 \} , \{ l = t|_{\mathrm{ker}(s)} , l'_2, l'_3 \}).
\end{align*}
By applying $\Phi$ to it, we get the compatible Lie $2$-algebra $\dot{\mathcal{C}} = (\dot{\mathcal{C}}, \dot{[~,~]}, \dot{J}, \dot{[~,~]'}, \dot{J'}, \dot{J^c})$, where the space of objects and the space of morphisms are respectively given by $\dot{\mathcal{C}}_0 = \mathcal{C}_0$ and $\dot{\mathcal{C}}_1 = \mathrm{ker}(s) \oplus \mathcal{C}_0$. Define $\beta_\mathcal{C} : \dot{\mathcal{C}} \rightarrow \mathcal{C}$ by $\beta_\mathcal{C} (h, x) = h + i(x)$. Then it is easy to verify that $\beta_\mathcal{C}$ is an isomorphism of compatible Lie $2$-algebras. This is infact a natural isomorphism. This completes the proof.
\end{proof}

\vspace*{0.5cm}
\noindent {\bf Acknowledgements.} Some parts of the work was carried out when the author was a postdoctoral fellow at Indian Institute of Technology (IIT) Kanpur. The author thanks the Institute for financial support.


%

\end{document}